{{
\documentclass[10pt, reqno]{amsart}
\usepackage{a4wide}
\usepackage[english, activeacute]{babel}
\usepackage{pifont}
\usepackage{amsmath,amsthm,amsxtra}

\usepackage{epsfig}
\usepackage{amssymb}

\usepackage{algorithm}
\usepackage{algpseudocode}
\usepackage{caption}
\usepackage{subcaption}

\usepackage{latexsym}
\usepackage{amsfonts}
\usepackage{dsfont}
\usepackage{hyperref}
\pagestyle{headings}
\usepackage{pdftricks}
\usepackage{xparse}
\usepackage{tabularx} 
\usepackage{multicol}
\usepackage{mathrsfs}
\usepackage{xcolor}
\usepackage{color}
\usepackage{hyperref}
\usepackage[export]{adjustbox}
\usepackage{pgf,tikz}
\usepackage{listings}

\newcommand{\R}{\mathbb{R}}
\newcommand{\N}{\mathbb{N}}

\newcommand{\E}{\mathbb{E}}
\newcommand{\PP}{\mathbb{P}}

\newcommand{\dist}{\operatorname{dist}}

\newcommand{\diam}{\operatorname{diam}}

\newtheorem{thm}{Theorem}[section]
\newtheorem{teo}{Theorem}[section]

\newtheorem{lem}[thm]{Lemma}

\newtheorem{ass}{Assumption}

\theoremstyle{remark}
\newtheorem{rem}{Remark}[section]

\newcommand{\be}{\begin{equation}}
\newcommand{\ee}{\end{equation}}
\newcommand{\bp}{\begin{proof}}
\newcommand{\ep}{\end{proof}}
\newcommand{\bel}{\begin{equation}\label}
\newcommand{\eeq}{\end{equation}}
\newcommand{\bea}{\begin{eqnarray}}
\newcommand{\eea}{\end{eqnarray}}
\newcommand{\bee}{\begin{eqnarray*}}
\newcommand{\eee}{\end{eqnarray*}}
\newcommand{\ben}{\begin{enumerate}}
\newcommand{\een}{\end{enumerate}}

\newcommand{\vertiii}[1]{{\left\vert\kern-0.25ex\left\vert\kern-0.25ex\left\vert #1 
\right\vert\kern-0.25ex\right\vert\kern-0.25ex\right\vert}}

\usepackage{hyperref}

\providecommand{\abs}[1]{\left|#1 \right|}
\providecommand{\norm}[1]{\left\| #1 \right\|}

\date{}

\title{A numerical approach for the fractional Laplacian via deep neural networks}

\author{Nicol\'as Valenzuela}
\address{Departamento de Ingenier\'ia Matem\'atica DIM, and CMM UMI 2807-CNRS, Universidad de Chile, Beauchef 851 Torre Norte Piso 5, Santiago Chile}
\email{nvalenzuela@dim.uchile.cl}
\thanks{N.V. is partially supported by Fondecyt no. 1191412, no. 1231250, Fondo Basal FB210005, ANID Fellowship 21231021, ANID Exploración 13220060 and {\color{black}Latin America Google PhD Fellowship Program}.}

\date{\today}

\numberwithin{equation}{section}

\begin{document}

\maketitle \markboth{Nicol\'as Valenzuela}{A numerical approach for the fractional Laplacian via deep neural networks}

\begin{abstract}

We consider the fractional elliptic problem with Dirichlet boundary conditions on a bounded and convex domain $D$ of $\R^d$, with $d \geq 2$. In this paper, we perform a stochastic gradient descent algorithm that approximates the solution of the fractional problem via Deep Neural Networks. Additionally, we provide four numerical examples to test the efficiency of the algorithm, and each example will be studied for many values of $\alpha \in (1,2)$ and $d \geq 2$.

\end{abstract}

\section{Introduction}\label{Sec:1}
{\color{black}\subsection{Motivation}
Deep Learning (DL) techniques \cite{Elb} have become principal actors on the numerical approximation of partial differential equation (PDE) solutions. Among them, linear PDEs \cite{Grohs, Nico}, semi-linear PDEs \cite{E, Hutz, Hutz2} and fully nonlinear PDEs \cite{Lye, Lye2} have been the main objectives of study. In particular, there exist numerical \cite{Berner1, Han} and theoretical \cite{Grohs, Hutz, Nico} results that suggest a good performance of Deep Neural Networks (DNNs) in the approximation of some PDE solutions. Previous references imply that DNNs can approximate (in suitable conditions) PDE solutions with arbitrary accuracy. Moreover, the requested DNN do not suffer of the so-called \emph{curse of dimensionality}, that is to say, the total number of parameters used to describe the DNN is at most polynomial on the dimension of the problem and on the reciprocal of the accuracy.}

\medskip
On the other hand, the fractional Laplacian has attracted considerable interest in past decades due to its vital role in some scientific disciplines such as physics, engineering and mathematics. This operator describes phenomena that exhibit anomalous behavior: {\color{black}is} used to describe some diffusion and propagation processes. Starting from the foundational work by Caffarelli and Silvestre \cite{CS}, the study of fractional problems has always required a great amount of detail and very technical mathematics. The reader can consult the monographs by \cite{Acosta,Bonito,Gulian,AK1,Lischke1}.

\medskip
{\color{black}In addition, Guilian, Raissi, Perdikaris and Karniadakis proposed in 2019 a machine learning approach for fractional differential equations \cite{Raissi}. This approach comes from the fact that for linear differential equation there exists a framework where the linear operator can be learned from Gaussian processes \cite{Raissi2}. The framework uses given data on the solution $u$ and the source term $f$, assuming that both functions are Gaussian processes with some parameters that depend on the linear operator. Once the operator is learned, it is possible to find the Gaussian processes that describe $u$ and $f$. Finally, the framework proposed in \cite{Raissi} is a generalization of the linear case.}


\medskip
In this paper we develop an  {\color{black} alternative} algorithm for approximating solutions of the high-dimensional fractional elliptic PDE with Dirichlet boundary condition over a bounded, convex domain $D \subset \R^d$, using deep learning techniques. To do so, we treat with {\color{black}the unique continuous solution of the fractional elliptic PDE. In 2018, Kyprianou, Osojnik and Shardlow proved in \cite{AK1} that the previously mentioned solution can be written in terms of the expected value of some stochastic processes. To obtain this result the authors assume hypotheses on the source and the boundary term, in addition with the setting on the domain $D$. With this and a relation between the fractional Laplacian and stochastic processes called isotropic $\alpha$-stable processes, the stochastic representation can be found.}\\

\medskip
{\color{black}That stochastic representation allows to give a Monte Carlo approximation of the solution for every point in the domain $D$. Each term of the Monte Carlo operator can in turn be approximated by a deep neural network. At the end, the average of the obtained DNNs is indeed an approximation of the solution of the PDE, and moreover, it is also a DNN.}


\medskip
The advantage of this methodology is that the deep neural network found can be evaluated in each point on $\R^d$, and there is not necessary to generate a Monte Carlo iteration for every point on $\R^d$, but in some of them for the training of the deep neural network. {\color{black} Furthermore, no data on the solution is needed in this approach since all the training data is generated in the algorithm, not so in the case of \cite{Raissi}.}

{\color{black}\subsection{Recent works on DL techniques} In the recent years different DL techniques involving DNNs and PDEs have been studied. Some of them are: Monte Carlo methods \cite{Grohs, Nico}, Multilevel Picard (MLP) iterations \cite{Beck, Hutz, Hutz2}, Forward Backward Stochastic Differential Equations (FBSDE) \cite{Castro, Han, Hure}, DeepONets \cite{CalderonDO,Chen,Lu} and Physic Informed Neural Networks (PINNs) \cite{Mishra, Mishra2}. Monte Carlo methods are used to approximate linear PDE solutions. The principal key is the \emph{Feynman-Kac representation} for some PDE solutions, where the solution can be written as the expectation of a random variable (or a stochastic process), and it is approximated by means of such random variable. Finally, each term of the mean can be approximated by a suitable DNN. 

\medskip

On other hand, MLP iterations are used in semilinear PDE solutions. Here, the Feynman-Kac representation can be seen as a functional with the PDE solution as the unique fixed point of such functional. Then, via \emph{Picard iterations}, the obtained terms are approximated by DNNs. Also, the PDEs are usually connected with stochastic differential equations by the \emph{Itô Formula}. The FBSDEs use this relation to obtain stochastic processes which are related with the PDE solution. Then, with a deep learning based algorithm is possible to have an approximation of the PDE solution. 

\medskip

Furthermore, there is also the DeepONets, which are approximations of \emph{infinite dimensional operators}, for instance, the operator that given a Dirichlet boundary condition, return the solution of a fixed PDE. The idea is to reduce the infinite dimensional problem to a finite dimensional space. Then, make a DNN between to finite dimensional spaces, and finally extend the DNN to the image set of the operator, which is infinite dimensional. From this technique, one of the most recent works is the approximation of the direct and inverse Calderón's mapping via DeepONets, given in \cite{CalderonDO}. Finally, but not less important, there are the PINNs, approximations via DNNƒs whose have in the loss function some constraints that comes from the nature of the PDE (such as the differential equation, boundary conditions, among others).}




\subsection*{Organization of this work} 
This work is organized as follows. In Section \ref{Sec:2} we introduce the necessary preliminary definitions and results for the understanding of the numerical problem. Section \ref{Sec:3} is devoted to the modeling of the numerical problem, such as the simulation of {\color{black}the involved} random variables, the generation of the training set, the use of the SGD algorithm to find an optimal deep neural network and the errors used to evaluate the accuracy of the algorithm. Section \ref{Sec:4} deals with four numerical examples, each one with a rigorous analysis. Finally, Section \ref{Sec:5} is devoted to the conclusions obtained in this work, and further discussion.

\section{Preliminaries}\label{Sec:2}

\subsection{Setting}\label{SSec:2.1}

\subsubsection{PDE Setting}\label{SSSec:2.1.1}

Along this work we will deal with the fractional Laplacian operator. Let $\alpha \in (0,2)$ and $d \geq 2$. The fractional Laplacian operator $(-\Delta)^{\frac{\alpha}2}$ is formally defined in $\R^d$ as:

\begin{equation}\label{FractLapOp}
	-(-\Delta)^{\alpha/2} u(x) = c_{d,\alpha} \lim_{\varepsilon \downarrow 0} \int_{\R^d \setminus B(0,\varepsilon)} \frac{u(y)-u(x)}{|y-x|^{d+\alpha}}dy, \hspace{.5cm} x \in \R^d,
\end{equation}

where {\color{black}$B(0,\varepsilon) = \{ x \in \R^d: \abs{x} < \varepsilon \}$}, $c_{d,\alpha}=- \frac{2^\alpha \Gamma((d+\alpha)/2)}{\pi^{d/2}\Gamma(-\alpha/2)}$ and $\Gamma(\cdot)$ is the classical Gamma function. For $d\geq 2$, let $D\subset \R^d$ be a convex and bounded domain, and consider the following Dirichlet boundary value problem

\begin{equation}\label{eq:LapFrac}
	\begin{cases}
		(-\Delta)^{\alpha/2} u = f &\quad \hbox{in} \; D, \\
		\hfill u = g &\quad \hbox{on} \; D^c,
	\end{cases}
\end{equation}

where we consider the following assumptions on the functions $f$ and $g$:
\begin{itemize}
	\item $g:D^c \to \R$ is a continuous function in $L^{1}_{\alpha}(D^c)$, $L_g >0$, that is to say
	\begin{equation}\label{Hg0}
		\int_{D^c} \frac{|g(x)|}{1+|x|^{d+\alpha}}dx < \infty. \tag{Hg-0}
	\end{equation}
	\item $f:D \to \R$ is a continuous function, such that
	\begin{equation}\label{Hf0}
		f \in C^{\alpha + \varepsilon_0}(\overline{D}) \qquad \hbox{for some fixed} \qquad \varepsilon_0>0. \tag{Hf-0}
	\end{equation}
\end{itemize}

\begin{rem}
	For the approximation Theorems in Subsection \ref{SSec:2.2}, we assume additionally that $g$ is a $L_g$-Lipschitz continuous function in Assumption \ref{Hg0}, and that $f$ is a $L_f$-Lipschitz continuous function in Assumption \ref{Hf0}.
\end{rem}

Under Assumptions \ref{Hg0}, \ref{Hf0}, Kyprianou{\color{black}, Osojnik and Shardlow} proved in \cite{AK1} that there exists a representation formula for the solution of \eqref{eq:LapFrac} via stochastic processes $(X_t)_{t\geq0}$ called isotropic $\alpha$-stable processes. The above is summarized in the next Theorem

\begin{teo}[\cite{AK1}, Theorem 6.1]
	Let $d \geq 2$ and assume that $D$ is a bounded domain in $\R^d$. Suppose that $g$ is a continuous function which belongs to $L^1_{\alpha}(D^c)$. Moreover, suppose that $f$ is a function in $C^{\alpha + \varepsilon}(\overline{D})$ for some $\varepsilon > 0$. Then there exists a unique continuous solution to \eqref{eq:LapFrac} in $L^1_{\alpha}(\R^d)$ given by
	\begin{equation}\label{eq:sol1}
		u(x)=\E_x[g(X_{\sigma_D})] + \E_x \left[\int_0^{\sigma_D} f(X_s)ds\right], \hspace{.5cm} x \in D,
	\end{equation}
	where $\sigma_D = \inf\{t>0: X_t \notin D\}$.
\end{teo}

An equivalent representation of \eqref{eq:sol1} can be stated in terms of the so-called Walk-on-Spheres (WoS) processes. We need to {\color{black}provide} additional objects that will be used in the new representation. First of all, the WoS process $(\rho_n)_{n \in \N_0}$ is defined as follows:
\begin{enumerate}
	\item $\rho_0 = X_0$,
	\item for all $n \in \N$, {\color{black}denote} $r_n = \dist(\rho_{n-1},\partial D)$. Let $Y_n$ be an independent copy of $X_{\sigma_{B(0,1)}}$. Then, $\rho_n$ is defined recursively as follows 
	\begin{equation}\label{eq:rec_rho}
	\rho_n = \rho_{n-1} + r_n Y_n.
	\end{equation}
	\item If $\rho_n \notin D$, let $N = n$ and stop the algorithm. Otherwise go to (2). 
\end{enumerate}

\begin{rem} The following remarks about the WoS processes can be stated:
	\begin{enumerate}
		\item The random variable $N$ is formally defined as $N = \min \{n \in \N_0 : \rho_n \notin D \}$. {\color{black}If $\alpha \in (0,2)$ and $D\subset \R^{d}$ is a convex, bounded domain then $N$} is a random variable finite a.s. \cite{AK1}.
		\item For further explanation of the WoS processes the reader can consult the work of Kyprianou et. al. \cite{AK1} and the paper of the author \cite{Nico}.
		\item It is easier to work with the WoS process instead of the $\alpha$-stable process for the reason that we do not need to simulate the entire trajectory of $(X_t)_{t\geq0}$, but a finite set of points of this trajectory.
	\end{enumerate}
\end{rem}

The next theorem gives the distribution of $X_{\sigma_{B(0,1)}}$, which allow us to simulate efficiently the copies $Y_n$ in the definition of the WoS processes. {\color{black}This theorem was proved in first instance in 1961 by Blumenthal, Getoor and Ray \cite{Blumenthal1}}

\begin{teo}[\cite{Blumenthal1,AK1}]\label{teo:Xball}
	Let $B(0,1)$ the unit ball centered at the origin and write $\sigma_{B(0,1)} = \inf \{t \geq 0 : X_t \notin B(0,1)\}$. Then
	\begin{equation}\label{eq:PXB01}
		\PP_0(X_{\sigma_{B(0,1)}} \in dy) = \pi^{-(d/2+1)}\Gamma(d/2)\sin(\pi \alpha/2) |1-|y|^2|^{-\alpha/2} |y|^{-d}dy, \quad |y|>1.
	\end{equation}
\end{teo}

On other hand, we define the expected occupation measure of the stable process prior to exiting a ball of radius $r>0$ centered in $x \in \R^d$ as follows:

\[
V_r(x,dy) := \int_0^{\infty} \PP_x\left(X_t \in dy, t<\sigma_{B(x,r)}\right)dt, \quad x\in \R^d, \quad |y|<1, \quad r>0.
\]

The next result is valid for $V_1(0,dy)$.

\begin{teo}[\cite{AK1}]\label{teo:V10}
	The measure $V_1(0,dy)$ is given for $|y|<1$ by
	\[
	V_1(0,dy) = 2^{-\alpha} \pi^{-d/2} \frac{\Gamma(d/2)}{\Gamma(\alpha/2)^2}|y|^{\alpha-d} \left( \int_0^{|y|^{-2} -1} (u+1)^{-d/2} u^{\alpha/2-1}du\right)dy.
	\]
\end{teo}

\begin{rem}\label{rem:V10}
	In \cite{Nico} we obtain the following equivalent form of $V_1(0,dy)$:
	\[
	V_1(0,dy) = 2^{-\alpha} \pi^{-d/2} \frac{\Gamma(d/2)}{\Gamma(\alpha/2)^2} B\left(\frac d2-\frac \alpha2,\frac \alpha2\right) |y|^{\alpha-d} \left(1-I\left(|y|^2;\frac d2- \frac \alpha2, \frac \alpha2\right)\right)dy,
	\]
	where $B(z,w)$ is the beta function
	\[
	B(z,w) := \int_0^1 u^{z-1} (1-u)^{w-1}du,
	\]
	which satisfies
	\[
	B(z,w) = \frac{\Gamma(z)\Gamma(w)}{\Gamma(z+w)},
	\]
	and $I(x;z,w)$, $x<1$ is the incomplete beta function
	\[
	I(x;z,w) := \frac{1}{B(z,w)} \int_0^x u^{z-1} (1-u)^{w-1}du.
	\]
\end{rem}

With an abuse of notation, {\color{black} for any bounded measurable function $f$ denote}
\[
V_r(x,f(\cdot)) = \int_{B(x,r)} f(y)V_r(x,dy).
\]
Now we are able to write the representation of $u$ in terms the WoS process, which proof is available in \cite{AK1}:

\begin{lem}[\cite{AK1}]
	For $x \in D$, $g \in L^{1}_{\alpha}(D^c)$ and $f \in C^{\alpha + \varepsilon_0}(\overline{D})$ we have the representation
	\begin{equation}\label{eq:sol2}
		u(x) = \E_x \left[g(\rho_N)\right] + \E_x \left[\sum_{n=1}^N r_n^{\alpha} V_1(0,f(\rho_{n-1}+r_n \cdot))\right].
	\end{equation} 
\end{lem}

Notice that we can {\color{black}determine} a probability measure over $B(0,1)$ from the measure $V_1(0,dy)$. Indeed, {\color{black}denote}
\begin{equation}\label{eq:kappa}
	\kappa_{d,\alpha} := \int_{B(0,1)} V_1(0,dy),
\end{equation}
then, $\mu$ defined as
\[
\mu(dy):= \kappa_{d,\alpha}^{-1} V_1(0,dy),
\]
is a probability measure over $B(0,1)$. With this normalization, the representation \eqref{eq:sol2} can be written as:

\begin{equation}\label{eq:solfinal}
	u(x) = \E_x \left[g(\rho_N)\right] + \E_x \left[ \sum_{n=1}^{N} r_n^{\alpha} \kappa_{d,\alpha} \E^{(\mu)}\left[f(\rho_{n-1} + r_n v)\right]\right],
\end{equation}

where $v$ is a random variable with density $\mu$ and $\E^{(\mu)}$ represent the expectation given by the probability measure $\mu$.

\subsubsection{DNN Setting}\label{SSSec:2.1.2}

For the identification of deep neural networks, we will use the same notation as in \cite{Hutz}. In particular, we consider the following setting for the architecture of a fully connected ReLu DNN. For $d \in \N$ define 
\[
A_d : \R^d \rightarrow \R^d,
\] 
the ReLu activation function such that for all $z \in \R^d$, $z=(z_1,...,z_d)$, with
\[
A_d(z)=(\max\{z_1,0\},...,\max\{z_d,0\}).
\] 
Let also
\begin{enumerate}
	\item[(NN1)] $H \in \N$ be the number of hidden layers; 
	\item[(NN2)] $(k_i)_{i=0}^{H+1}$ be a positive integer sequence; 
	\item[(NN3)] $W_{i} \in \R^{k_i \times k_{i-1}}$, $B_{i} \in \R^{k_i}$, for any $i=1,...,H+1$ be the weights and biases, respectively;
	\item[(NN4)] $x_0 \in \R^{k_0}$, and for $i = 1,...,H$ let 
	\begin{equation}\label{x_i}
		x_i = A_{k_i}(W_ix_{i-1}+B_i).
	\end{equation}
\end{enumerate}
We call 
\begin{equation}\label{eq:DNN_def}
	\Phi := (W_i,B_i)_{i=1}^{H+1} \in \prod_{i=1}^{H+1} \left(\R^{k_i \times k_{i-1}} \times \R^{k_i}\right),
\end{equation}
the DNN associated to the parameters in (NN1)-(NN4). The space of all DNNs in the sense of \eqref{eq:DNN_def} is going to be denoted by ${\bf N}$, namely
\[
{\bf N} = \bigcup_{H \in \N} \bigcup_{(k_0,...,k_{H+1})\in \N^{H+2}} \left[\prod_{i=1}^{H+1} \left(\R^{k_i \times k_{i-1}} \times \R^{k_i}\right)\right].
\]
Define the realization of the DNN $\Phi \in {\bf N}$ as
\begin{equation}\label{Realization}
	\mathcal{R}(\Phi)(x_0) =  W_{H+1}x_H + B_{H+1}.
\end{equation}
Notice that $\mathcal{R}(\Phi) \in C(\R^{k_0},\R^{k_{H+1}})$. For any $\Phi \in {\bf N}$ define
\begin{equation}\label{P_D}
	\mathcal{P}(\Phi) = \sum_{n	=1}^{H+1} k_n (k_{n-1}+1), \qquad \mathcal{D}(\Phi) = (k_0,k_1,...,k_{H+1}),
\end{equation}
and
\begin{equation}\label{norma}
	\vertiii{  \mathcal{D}(\Phi )  } = \max\{ k_0,k_1,...,k_{H+1} \}.
\end{equation}
The entries of $(W_i,B_i)_{i=1}^{H+1}$ will be the weights and biases of the DNN, $\mathcal{P}(\Phi)$ represents the total number of parameters used to describe the DNN, working always with fully connected DNNs, and $\mathcal{D}(\Phi)$ representes the dimension of each layer of the DNN. Notice that $\Phi \in {\bf N}$ has $H+2$ layers: $H$ of them hidden, one input and one output layer.

\subsection{Approximation Theorem}\label{SSec:2.2} 

In \cite{Nico} it has been proved that the solution $u$ of the problem \eqref{eq:LapFrac}, written as \eqref{eq:solfinal}, can be approximated via ReLu DNNs with arbitrary accuracy. In order to obtain that result, additional assumptions on the functions $f,g$ and on the structure of the domain $D$ need to be stated. In particular, we need the following assumptions:

\begin{ass}\label{Sup:g}
	{\color{black} Let $d\geq2$, $D \subset \R^d$ be a convex and bounded domain. Let $a,b \geq 1$, $p \in (1,\alpha)$, $B>0$ some constants which do not depend on $d$}.  Let $g:D^c\to \R$ satisfying \eqref{Hg0}.{\color{black} Assume that for any $\delta_g \in (0,1]$, the function $g$ can be approximated by a ReLu DNN $\Phi_g$ which satisfies}
	\begin{enumerate}
		\item $\mathcal R(\Phi_{g}):D^c \to \R$  is continuous, and
		\item The following are satisfied:
		\begin{align}
			|g(y)-\left(\mathcal{R}(\Phi_{g})\right)(y)| &\leq\delta_g Bd^p(1+|y|)^{p}, \hspace{.5cm} \forall y \in D^c. \tag{Hg-1} \label{H1}\\
			|\left(\mathcal{R}(\Phi_{g})\right)(y)| &\leq Bd^p(1+|y|)^{p}, \hspace{1.0cm} \forall y \in D^c. \tag{Hg-2} \label{H2}\\
			\vertiii{\mathcal{D}(\Phi_{g})} &\leq B d^b \delta_g^{-a}. \tag{Hg-3} \label{H3}
		\end{align}
	\end{enumerate}
\end{ass}

\begin{ass}\label{Sup:D}
	{\color{black}Let $d\geq2$}. Let $\alpha\in (1,2)$, $a,b \geq 1$ and $B>0$ {\color{black} as in Assumption \ref{Sup:g}}. Suppose that $D${\color{black}$\subset\R^d$ is a convex,} bounded domain that enjoys the following structure:
	\begin{enumerate}
		\item For any $\delta_{\dist} \in (0,1]$, the function $x \mapsto \dist(x,\partial D)$ can be approximated by a ReLu DNN $\Phi_{\dist} \in {\bf N}$ such that 
		\[
		\sup_{x \in D} \left|\dist(x,\partial D) - \left(\mathcal{R}(\Phi_{\dist})\right)(x)\right| \leq \delta_{\dist}, \tag{HD-1} \label{HD-1} 
		\]
		and
		\[
		\vertiii{\mathcal{D}(\Phi_{\dist})} \leq Bd^b\lceil \log(\delta_{\dist}^{-1}) \rceil^{a}. \tag{HD-2} \label{HD-2}
		\]
		\item For all $\delta_\alpha \in (0,1)$ there exists a ReLu DNN $\Phi_{\alpha} \in {\bf N}$, {\color{black} $\mathcal R(\Phi_{\alpha}) \in C(\R,\R) $} such that
		\begin{equation}
			\sup_{|x|\leq\diam(D)}\left|\left(\mathcal{R}(\Phi_{\alpha})\right)(x) -x^{\alpha}\right| \leq \delta_{\alpha}, \tag{HD-3} \label{HD-3}
		\end{equation}
		and
		\[
		\vertiii{\mathcal{D}(\Phi_{\alpha})} \leq Bd^b \delta_{\alpha}^{-a}. \tag{HD-4} \label{HD-4}
		\]
		Moreover, $\mathcal{R}(\Phi_{\alpha})$ is a $L_{\alpha}$-Lipschitz function, $L_{\alpha}>0$, for $|x|\leq \diam(D)$.
	\end{enumerate}
\end{ass} 

\begin{ass}\label{Sup:f}
	Let $d \geq 2$, {\color{black} $D\subset\R^d$ be a convex and bounded domain}. {\color{black} Let $a,b \geq 1$ and $B>0$ as in Assumption \ref{Sup:g}}. Let $f: D \to \R$  a function satisfying \eqref{Hf0}. {\color{black} Assume that for any} $\delta_f \in (0,1)$, {\color{black} the function $f$ can be approximated by a ReLu DNN $\Phi_f \in {\bf N}$ that satisfies}
	\begin{enumerate}
		\item $\mathcal{R}(\Phi_f):D \to \R$ is $\widetilde{L}_f$-Lipschitz continuous, $\widetilde{L}_f>0$, and
		\item The following are satisfied:
		\begin{align}
			|f(x) - \left(\mathcal{R}(\Phi_f)\right)(x)| &\leq \delta_f, \qquad x \in D. \tag{Hf-1} \label{H5}\\
			\vertiii{\mathcal{D}(\Phi_f)} &\leq Bd^b\delta_f^{-a}. \tag{Hf-2} \label{H6} 
		\end{align} 
	\end{enumerate}
\end{ass}

With these three assumptions, one can state the following Theorem, {\color{black}as in \cite{Nico}}:

\begin{thm}\label{Main} 
	{\color{black} Let $d \geq 2$, $D\subset \R^d$ be a convex and bounded domain,} $\alpha \in (1,2)$, {\color{black}$a,b\geq 1$, $p \in (1,\alpha)$ as in Assumption \ref{Sup:g} and let $q \in \left(1,\frac{\alpha}{p}\right)$}. Assume that \eqref{Hg0} and \eqref{Hf0} are satisfied. Suppose that for every $\delta_{\alpha}, \delta_{\dist}, \delta_{f}, \delta_{g} \in (0,1)$ there exist ReLu DNNs $\Phi_g$, $\Phi_{\alpha}, \Phi_{\dist}$ and $\Phi_f$ satisfying Assumptions \ref{Sup:g}, \ref{Sup:D} and \ref{Sup:f}, respectively. Then there exist $\hat{B},\eta>0$ such that for every $\epsilon \in (0,1]$, the solution of \eqref{eq:LapFrac} can be approximated by a ReLu DNN $\Psi_{\epsilon}$ such that the realization $\mathcal{R}(\Psi_{\epsilon}): D \to \R$ is a continuous function which satisfies:
	\begin{enumerate}
		\item Proximity in $L^q(D)$: If $u$ is the solution of \eqref{eq:LapFrac}
		\begin{equation}
			\left(\int_D \left|u(x) - \left(\mathcal{R}(\Psi_{\epsilon})\right)(x)\right|^q dx\right)^{\frac 1q} \leq \epsilon.
		\end{equation}
		\item Bounds:
		\begin{equation}
			\mathcal{P}(\Psi_{\epsilon}) \leq \widehat{B}|D|^{\eta}d^{\eta} \epsilon^{-\eta}.
		\end{equation}
		The constant $\widehat{B}$ depends on $\norm{f}_{L^{\infty}(D)}$, the Lipschitz constants of $g$, $\mathcal{R}(\Phi_f)$ and $\mathcal{R}(\Phi_{\alpha})$, and on $\diam(D)$. {\color{black} The constant $\eta$ depends on $a,b,p$ and $q$.}
	\end{enumerate}
\end{thm} 

{\color{black}\begin{rem}
	This paper tries to face this problem in a numerical way. We will propose an algorithm to find the ReLu DNN that exists from Theorem \ref{Main} without the curse of dimensionality.
\end{rem}}

\begin{rem}
	Convexity is necessary for the law of large numbers, in order to do a well approximation of the solution via Monte Carlo iterations.
\end{rem}

\section{Numerical problem modeling}\label{Sec:3}

In this section we propose an algorithm that finds a ReLu Deep Neural Network that fits the solution of \eqref{eq:LapFrac}. For this, we need preliminary work that we explain in each subsection. In particular, we need to compute efficiently each term involved in the representation \eqref{eq:solfinal}.

\subsection{Simulation of random variables}\label{SSec:3.1}

First of all, we will give an analytical value for the expression $\kappa_{d,\alpha}$ {\color{black}defined in \eqref{eq:kappa}}. The next Lemma talk about this quantity.

\begin{lem}\label{lem:kappa}
	For any $d \geq 2$ and $\alpha \in (0,2)$,
	\[
	\kappa_{d,\alpha} = \frac{2^{1-\alpha}}{\alpha} \frac{\Gamma(d/2)}{\Gamma(\alpha/2)\Gamma(d/2+\alpha/2)}.
	\]
\end{lem}

\begin{proof}
	From Remark \ref{rem:V10} and the definition of $\kappa_{d,\alpha}$ in \eqref{eq:kappa} is easy to see that
	\[
	\kappa_{d,\alpha} = 2^{-\alpha} \pi^{-d/2} \frac{\Gamma\left(\frac d2 - \frac \alpha 2\right)}{\Gamma \left(\frac \alpha 2\right)} \int_{B(0,1)} |y|^{\alpha-d} \left(1 - I \left(|y|^2; \frac d2 - \frac \alpha 2, \frac \alpha 2\right)\right)dy.
	\]
	Using spherical coordinates we have
	\[
	\kappa_{d,\alpha} =  2^{-\alpha} \pi^{-d/2} \frac{\Gamma\left(\frac d2 - \frac \alpha 2\right)}{\Gamma \left(\frac \alpha 2\right)} \int_{\mathbb{S}^{d-1}} \int_0^1 r^{\alpha - d} r^{d-1} \left(1 - I\left(r^2; \frac d2 -\frac \alpha 2, \frac \alpha 2\right)\right)dr dS,
	\]
	where $\mathbb{S}^{d-1}$ is the surface area of the unit $(d-1)$-sphere embedded in dimension $d$. {\color{black}This} measure is given by
	\[
	\left|\mathbb{S}^{d-1}\right| = \frac{2\pi^{d/2}}{\Gamma \left(\frac{d}{2}\right)}.
	\]
	Then 
	\[
	\begin{aligned}
		\kappa_{d,\alpha} &= 2^{-\alpha} \pi^{-d/2} \frac{\Gamma\left(\frac d2 - \frac \alpha 2\right)}{\Gamma \left(\frac \alpha 2\right)} \left|\mathbb{S}^{d-1}\right| \int_0^1 r^{d-1} \left(1 - I\left(r^2; \frac d2 - \frac \alpha 2, \frac \alpha 2\right)\right)dr \\
		&= 2^{1-\alpha} \frac{\Gamma\left(\frac d2 -\frac \alpha 2\right)}{\Gamma\left(\frac \alpha 2\right) \Gamma \left(\frac d2\right)} \left(\frac 1\alpha - \int_0^1 r^{\alpha-1} I\left(r^2;\frac d2 - \frac \alpha 2, \frac \alpha 2\right)dr\right).
	\end{aligned}
	\]
	{\color{black}To work} with the above integral, denote
	\[
	J = \int_0^1 r^{\alpha-1} I\left(r^2;\frac d2 - \frac \alpha 2, \frac \alpha 2\right)dr.
	\]
	In $J$ we can use integration by parts to obtain
	\[
	J = \frac 1\alpha - \frac{1}{\alpha B\left(\frac d2 - \frac \alpha 2, \frac \alpha 2\right)} \int_0^1 r^{d-2}(1-r^2)^{\frac \alpha 2 - 1} 2rdr.
	\]
	Then, by change of variables $\tilde{r}= r^2$; $d\tilde{r} = 2rdr$ and the definition of Beta function it follows that
	\[
	\begin{aligned}
		J &= \frac 1\alpha - \frac{1}{\alpha B\left(\frac d2 - \frac \alpha 2, \frac \alpha 2\right)} \int_0^1 \tilde{r}^{\frac d2 -1} (1-\tilde{r})^{\frac \alpha 2 - 1}d\tilde{r} \\
		&=\frac 1\alpha - \frac{B\left(\frac d2, \frac \alpha 2\right)}{\alpha B\left(\frac d2 - \frac \alpha 2, \frac \alpha 2\right)} \\
		&= \frac 1\alpha - \frac{\Gamma\left(\frac d2\right)}{\alpha\Gamma\left(\frac d2 - \frac \alpha 2\right)\Gamma\left(\frac \alpha 2\right)} \frac{\Gamma\left(\frac d2\right)\Gamma\left(\frac \alpha 2\right)}{\Gamma\left(\frac d2 + \frac \alpha 2\right)}.
	\end{aligned}
	\]
	From this we obtain
	\[
	J = \frac 1\alpha - \frac{\Gamma\left(\frac d2\right)^2}{\alpha \Gamma\left(\frac d2 - \frac \alpha 2 \right)\Gamma\left(\frac d2 + \frac \alpha 2 \right)}.
	\]
	Replacing this value in $\kappa_{d,\alpha}$ we conclude that
	\[
	\kappa_{d,\alpha} = \frac{2^{1-\alpha}}{\alpha} \frac{\Gamma\left(\frac d2 \right)}{\Gamma\left(\frac \alpha 2 \right)\Gamma\left(\frac d2 + \frac \alpha 2 \right)}.
	\]
\end{proof}

Now we will work with each random variable of the representation \eqref{eq:solfinal}. First we will give a scheme to simulate $X_{\sigma_{B(0,1)}}$ and find an algorithm to simulate a WoS process $(\rho_n)_{n=0}^{N}$ starting at a point $x \in D$. Then we make a scheme in order to do the simulations of the random variables $v$ with probability measure $\mu$. In both random variables, we will notice that the densities only depend on the radius of the initial point, and then one can simulate those random variables using the marginal radial density of each random variable.

\subsubsection{Simulation of $X_{\sigma_{B(0,1)}}$}\label{SSSec:3.1.1}

Denote by $f_{R_X}$ the marginal radial density of $X_{\sigma_{B(0,1)}}$. From Theorem \ref{teo:Xball} we can obtain $f_{R_X}(r)$ explicitly for $r>1$. Indeed, for all $r>1$ by spherical coordinates one has
\[
\begin{aligned}
	f_{R_X}(r) &= \pi^{-\left(\frac d2 + 1\right)}\Gamma\left(\frac d2\right) \sin\left(\frac{\pi\alpha}{2}\right)\int_{\mathbb{S}^{d-1}} (r^2 - 1)^{-\frac \alpha 2} r^{-d} r^{d-1} dS dr\\
	&= \pi^{-\left(\frac d2 + 1\right)}\Gamma\left(\frac d2\right) \sin\left(\frac{\pi\alpha}{2}\right) \left|\mathbb{S}^{d-1}\right| (r^2 - 1)^{-\frac \alpha 2} r^{-1} dr.
\end{aligned}
\] 
Then, for any $r>1$ 
\[
f_{R_X}(r) = \frac 2\pi \sin\left(\frac{\pi\alpha}2\right)(r^2-1)^{-\frac \alpha 2} r^{-1}dr.
\]
From the density $f_{R_X}$, one can calculate the distribution $F_{R_X}(r)$ for $r>1$. Indeed
\[
\begin{aligned}
	F_{R_X}(r) &= \int_{1}^{r} f_{R_X}(s)ds\\
	&= \frac{2}{\pi} \sin\left(\frac{\pi \alpha}{2}\right) \int_1^r (s^2-1)^{-\frac{\alpha}{2}} s^{-1}ds.
\end{aligned}
\]
By a change of variables $t = \frac{1}{s}$, one obtains
\[
\begin{aligned}
	F_{R_X}(r) &= \frac{2}{\pi} \sin\left(\frac{\pi \alpha}{2}\right) \int_{1/r}^1 \left(\frac{1}{t^2}-1\right)^{-\frac{\alpha}{2}}t^{-1}dt\\
	&= \frac{2}{\pi} \sin\left(\frac{\pi \alpha}{2}\right) \int_{1/r}^1 (1-t^2)^{-\alpha/2} t^{\alpha-1}dt.
\end{aligned}
\]
One can do a new change of variables, of the form $s=t^2$ to obtain
\[
\begin{aligned}
	F_{R_X}(r) &= \frac{2}{\pi} \sin\left(\frac{\pi \alpha}{2}\right) \int_{1/r^2}^1 (1-s)^{1-\alpha/2-1}\frac{s^{\alpha/2-1}}2 ds\\
	&= \frac{1}{\pi} \sin\left(\frac{\pi \alpha}{2}\right) B\left(\frac{\alpha}2,1-\frac{\alpha}2\right) \left(1-I\left(\frac{1}{r^2};\frac{\alpha}2,1-\frac{\alpha}2\right)\right),
\end{aligned}
\]
where we have used the definitions of Beta and Incomplete Beta functions. Therefore  $F_{R_X}(r)$ can be written, for $r>1$, as
\[
F_{R_X}(r) = \frac{1}{\pi} \sin\left(\frac{\pi \alpha}{2}\right) \Gamma\left(\frac{\alpha}2\right)\Gamma\left(1-\frac{\alpha}2\right) \left(1-I\left(\frac{1}{r^2};\frac{\alpha}2,1-\frac{\alpha}2\right)\right).
\]

Notice that $F_{R_X}(r)$ has an explicit inverse by using the well known inverse of the Incomplete Beta function $I^{-1}$. First of all notice that for all $\alpha \in (0,2)$, the Euler's reflection formula \cite{reflection} is valid:
\[
\Gamma\left(\frac{\alpha}2\right)\Gamma\left(1-\frac{\alpha}2\right) = \frac{\pi}{\sin\left(\frac{\pi \alpha}{2}\right)}.
\]
Then, $F_{R_X}(r)$ can be simply written as
\[
F_{R_X}(r) = 1 - I\left(\frac 1{r^2}; \frac{\alpha}2, 1 - \frac{\alpha}2\right).
\]
Therefore,
\[
\begin{aligned}
	u = 1 - I\left(\frac{1}{F^{-1}_{X_R}(u)^2};\frac{\alpha}2,1-\frac{\alpha}2\right)&\Longrightarrow I\left(\frac{1}{F^{-1}_{X_R}(u)^2};\frac{\alpha}2,1-\frac{\alpha}2\right) = 1 - u\\
	&\Longrightarrow \frac{1}{F^{-1}_{X_R}(u)^2} = I^{-1}\left(1-u; \frac{\alpha}2, 1- \frac{\alpha}2\right).
\end{aligned}
\]
This implies that
\begin{equation}\label{eq:F-1}
	\displaystyle F^{-1}_{X_R}(u) = \frac{1}{I^{-1}\left(1-u; \frac{\alpha}2, 1-\frac{\alpha}2\right)^{\frac 12}}.
\end{equation}
The fact that $F_{R_X}^{-1}$ can be expressed explicitly implies that a copy of a real valued random variable with density $f_{R_X}$ can be simulated as $F^{-1}_{R_X}(U)$, where $U \sim \hbox{Uniform}(0,1)$.

\medskip
By the radial symmetry of the process $X_{\sigma_{B(0,1)}}$, one can state the following algorithm for the simulation of copies of $X_{\sigma_{B(0,1)}}$:

\begin{algorithm}[]
	\caption{Simulation of copies of $X_{\sigma_{B(0,1)}}$}\label{alg:Xball}
	\begin{algorithmic}
		\State \textbf{Input:} $\alpha \in (0,2),$ $d\geq 2$
		\State \textbf{Output:} $Y$ a copy of $X_{\sigma_{B(0,1)}} \in \R^d$
		\State Simulate $U \sim \hbox{Uniform}(0,1)$
		\State Simulate $\Theta_1,...,\Theta_{d-2} \sim \text{Uniform}(0,\pi)$ and $\Theta_{d-1} \sim \text{Uniform}(0,2\pi)$
		\State $R \gets F_{R_X}^{-1}(U)$ as in \eqref{eq:F-1}
		\State $w_1 \gets \cos(\Theta_1)$
		\For{$i=2,...,d-1$}
		\State $w_i \gets \cos(\Theta_i)\prod_{k=1}^{i-1} \sin(\Theta_k)$
		\EndFor
		\State $w_d \gets \prod_{k=1}^{d-1} sin(\Theta_k)$
		\State $w \gets (w_1,...,w_d)$
		\State $Y \gets Rw$
	\end{algorithmic}
\end{algorithm}
\begin{rem}
	For Algorithm \ref{alg:Xball}, the point $w$ generated is a random point on $B(0,1)\subset \R^d$, chosen uniformly.
\end{rem}

Algorithm \ref{alg:Xball} and recursion \eqref{eq:rec_rho} allow us to write an algorithm in order to obtain a copy of the WoS process $(\rho_n)_{n=0}^{N}$, starting at a point $x \in D$.
\begin{algorithm}[]
	\caption{Simulation of copy of WoS process starting at $x \in D$}\label{alg:rho}
	\begin{algorithmic}
		\State \textbf{Input:} $\alpha \in (0,2),$ $d\geq 2$, $x \in D$
		\State \textbf{Output:} $(\rho_n)_{n=0}^{N}$ a copy of the WoS process, $(r_n)_{n=1}^{N}$ the sequence of radius
		\State $\rho_0 \gets x$
		\State $n \gets 0$
		\While{$\rho_n \in D$}
		\State $n \gets n + 1$
		\State $r_n \gets \dist(\rho_{n-1},\partial D)$
		\State Simulate $Y$ with Algorithm \ref{alg:Xball}.
		\State $\rho_n \gets \rho_{n-1} + r_n Y$
		\EndWhile
		\State $N \gets n$
	\end{algorithmic}
\end{algorithm}

\subsubsection{Simulation of $v$} 
Denote by $f_{R_v}$ the marginal radial density of the random variable $v$ with probability measure $\mu$. From Theorem \ref{teo:V10} and Remark \ref{rem:V10} we obtain $f_{R_v}(r)$ explicitly for $r<1$. Indeed, for all $r<1$, by spherical coordinates one has
\[
\begin{aligned}
	f_{R_v}(r) &= \frac{2^{-\alpha}\pi^{-\frac d2}}{\kappa_{d,\alpha}} \frac{\Gamma\left(\frac d2\right)}{\Gamma\left(\frac{\alpha}2\right)^2}B\left(\frac d2 - \frac{\alpha}2,\frac{\alpha}2\right)\int_{\mathbb{S}^{d-1}} r^{\alpha-d}\left(1-I\left(r^2;\frac d2 -\frac{\alpha}2,\frac{\alpha}2\right)\right)r^{d-1}dSdr\\
	&=  \frac{2^{-\alpha}\pi^{-\frac d2}}{\kappa_{d,\alpha}} \frac{\Gamma\left(\frac d2\right)}{\Gamma\left(\frac{\alpha}2\right)^2}B\left(\frac d2 - \frac{\alpha}2,\frac{\alpha}2\right) |\mathbb{S}^{d-1}|r^{\alpha-1}\left(1-I\left(r^2;\frac d2 -\frac{\alpha}2,\frac{\alpha}2\right)\right)dr.
\end{aligned}
\]
Recall the value of $\kappa_{d,\alpha}$ from Lemma \ref{lem:kappa} and $|\mathbb{S}^{d-1}| = \frac{2\pi^{\frac d2}}{\Gamma \left(\frac d2\right)}$. Therefore 

\[
\begin{aligned}
	f_{R_v}(r) &= \alpha\frac{\Gamma\left(\frac d2 + \frac{\alpha}2\right)}{\Gamma\left(\frac d2\right)\Gamma\left(\frac{\alpha}2\right)} B\left(\frac d2 - \frac{\alpha}2, \frac{\alpha}2\right)r^{\alpha-1}\left(1-I\left(r^2;\frac d2 -\frac{\alpha}2,\frac{\alpha}2\right)\right)dr\\
	&= \alpha \frac{B\left(\frac d2 - \frac{\alpha}2, \frac{\alpha}2\right)}{B\left(\frac d2, \frac{\alpha}2\right)} r^{\alpha-1}\left(1-I\left(r^2;\frac d2 -\frac{\alpha}2,\frac{\alpha}2\right)\right)dr.
\end{aligned}
\]
Now we can calculate the distribution $F_{R_v}(r)$ for $r<1$. Indeed
\[
\begin{aligned}
	F_{R_v}(r) &= \int_{0}^r f_{R_V}(s)ds\\
	& = \alpha \frac{B\left(\frac d2 - \frac{\alpha}2, \frac{\alpha}2\right)}{B\left(\frac d2, \frac{\alpha}2\right)} \left(\frac{r^{\alpha}}{\alpha}-\int_{0}^r s^{\alpha-1}I\left(s^2;\frac d2 -\frac{\alpha}2,\frac{\alpha}2\right)ds\right).
\end{aligned}
\]
Denote $J_1 = \int_{0}^r s^{\alpha-1}I\left(s^2;\frac d2 -\frac{\alpha}2,\frac{\alpha}2\right)ds$. We integrate by parts $J$. Following the notation $u = I\left(s^2;\frac d2 -\frac{\alpha}2,\frac{\alpha}2\right)$, with $du = \frac{2}{B\left(\frac d2 - \frac{\alpha}2, \frac{\alpha}2\right)} s^{d-\alpha-1}(1-s^2)^{\frac{\alpha}2-1}ds$ and $dv = s^{\alpha-1}ds$ with $v = \frac{s^{\alpha}}{\alpha}$, one obtains
\[
J_1 = \frac{r^{\alpha}}{\alpha}I\left(r^2;\frac d2 - \frac{\alpha}2, \frac{\alpha}2\right) - \frac{2}{\alpha B\left(\frac d2 - \frac{\alpha}2, \frac{\alpha}2\right)} \int_0^r s^{d-1}(1-s^2)^{\frac{\alpha}2-1}ds.
\]
Denote $J_2 = \int_0^r s^{d-1}(1-s^2)^{\frac{\alpha}2-1}ds$. With a change of variables $u=s^2$, $J_2$ can be written as
\[
J_2 = \frac 12 \int_0^{r^2} u^{\frac d2 - 1}(1-u)^{\frac{\alpha}2-1}du = \frac{B\left(\frac d2, \frac{\alpha}2\right)}{2} I\left(r^2;\frac d2, \frac{\alpha}2\right).
\]
Therefore
\[
J_1 = \frac{r^{\alpha}}{\alpha}I \left(r^2,\frac d2 - \frac{\alpha}2,\frac{\alpha}2\right) - \frac 1{\alpha} \frac{B\left(\frac d2, \frac{\alpha}2\right)}{B\left(\frac d2 - \frac{\alpha}2,\frac{\alpha}2\right)}I\left(r^2;\frac d2, \frac{\alpha}2\right).
\]
Finally, for $r<1$, the distribution $F_{R_{v}}(r)$ is
\[
F_{R_v}(r) = I\left(r^2;\frac d2, \frac{\alpha}2\right) + r^{\alpha}\frac{B\left(\frac d2 - \frac{\alpha}2, \frac{\alpha}2\right)}{B\left(\frac d2, \frac{\alpha}2\right)}\left(1 - I\left(r^2;\frac d2 - \frac{\alpha}2, \frac{\alpha}2\right)\right).
\]
Due the form of $F_{R_v}(r)$, it is difficult to calculate its inverse $F^{-1}_{R_v}$. Instead, by the knowledge of $f_{R_v}$ and $F_{R_v}$ we will simulate random variables with density $f_{R_v}$ using the Newton-Raphson method:

\begin{algorithm}
	\caption{Newton-Raphson algorithm}\label{alg:NR}
	\begin{algorithmic}
		\State \textbf{Input:} $r,\delta \in (0,1)$
		\State \textbf{Output:} $R$ a copy of a random variable with density $f_{R_v}$
		\State Simulate $U \sim \hbox{Uniform}(0,1)$
		\State $R \gets r$
		\While{$|F_{R_v}(R)-U|\geq \delta$}
		\State $R \gets R - \frac{F_{R_v}(R)-U}{f_{R_v}(R)}$
		\EndWhile
	\end{algorithmic}
\end{algorithm}

\medskip
It is important to recall that the tolerance $\delta$ must be small to have an improved copy. For the next numerical examples, we set $\delta = 10^{-3}$. The simulation of copies of a random variable $V$ with density $\mu$ is a quite similar to the case of $X_{\sigma_{B(0,1)}}$. We will use the same way to generate the random vector $w \in \R^d$ and the fact that the copy of $v$ is the multiplication of $R$ and $w$, with the difference that $R$ will be a copy obtained from Newton-Raphson algorithm.

\subsection{Monte Carlo training set generation}\label{SSec:3.2}

To obtain train data, we will use Monte Carlo simulations. Let $M,P \in \N$. Choose randomly $P$ points in a set {\color{black}containing} $D$, namely $(x_k)_{k=1}^{P}$. Each value $u(x_k)$ will be approximated by a Monte Carlo simulation with $M$ iterations. The advantage of use Monte Carlo in the train data is that we can create the data even when the solution $u$ is unknown.

\medskip
Denote by $\mathbb{D}$ the training set. $\mathbb{D}$ can be written as
\begin{equation}\label{eq:traindata}
	\mathbb{D} = \left\{(x_k,\hat{u}_k): k \in \{1,...,P\}\right\} \subset \R^d \times \R,
\end{equation}
where for each $k=1,...,P$, $\hat{u}_k$ has the form
\begin{equation}\label{eq:uk}
	\hat{u}_k = \begin{cases}
		\displaystyle\frac{1}{M} \sum_{i=1}^{M} \left(g\left(\rho^{i}_{N_i}\right) + \sum_{n=1}^{N_i} \kappa_{d,a} r_n^{\alpha} f\left(\rho^{i}_n + r^{i}_n v_{i}\right)\right) & \hbox{if } x_k \in D,\\
		g(x_k) & \hbox{if } x_k \notin D.
	\end{cases}
\end{equation}
Here, $\left((\rho^{i}_n,r_n^i)_{n=0}^{N_i}\right)_{i=1}^{M}$ are $M$ copies of the WoS process starting at $x_k$, with their respective radius, and $(v_i)_{i=1}^{M}$ are $M$ copies of the random variable $v$ with probability measure $\mu$.

\subsection{DNN approximation}\label{SSec:3.3}

Recall from Theorem \ref{Main} the existence of a ReLu DNN that approximates the solution of PDE \eqref{eq:LapFrac}. This DNN may have different size of the hidden layers $H$ or vector $\mathcal{D}(\Phi)$ for different settings of the problem \eqref{eq:LapFrac}.

\medskip
In the next simulations, in order to simplify the numerical computation, we set a fixed number of hidden layers $H$ and their dimensions. We will find the optimal DNN that approximates the solution $u$ of problem \eqref{eq:LapFrac} for the previous conditions. In order to do this, we set an initial ReLu DNN $\Phi$, and we train that DNN using the stochastic gradient descent (SGD) method. The training of the DNN is made with a loss function that depends on the training set defined in \eqref{eq:traindata}. In the Sections below we talk about each one in detail.

\subsubsection{DNN generation}\label{SSSec:3.3.1}
Under the notation in Section \ref{SSSec:2.1.2}, we fix $H=7$ with $k_i = 110$ for all $i=1,...,H$. This will be the number of hidden layers and the dimension of such layers. We set also $k_0 = d$ to be the dimension of the input layer and $k_{H+1}=1$ to be the dimension of the output layer. 

\medskip
With this setting, we generate an initial ReLu DNN $\Phi=((W_i,b_i))_{i=1}^{8}$, where
\begin{itemize}
	\item $W_1 \in \R^{110\times d}$, $b_1 \in \R^{110}$;
	\item For all $i=1,..,7$, $W_i \in \R^{110\times 110}$, $b_i \in \R^{110}$;
	\item $W_8 \in \R^{1\times 110}$, $b_8 \in \R$.
\end{itemize}

The values of $(W_i,b_i)_{i=1}^{8}$ are chosen in an arbitrary way, and they will become parameters of an optimization problem. The generation of the ReLu DNN in Python is made with the \emph{Tensorflow} library.

\subsubsection{SGD algorithm implementation}\label{SSSec:3.3.2}
In order to find the optimal values of the parameters $(W_i,b_i)_{i=1}^{8}$ we perform a SGD algorithm. Consider a training set with the form of \eqref{eq:traindata} of size $|\mathbb{D}|= P \in \N$ and let $N_{Iter} \in \N$, $\gamma \in (0,1)$ be the number of iterations of the SGD algorithm and its learning rate, respectively.

\medskip
For each $j = 1,...,N_{Iter}$ take $S_j$ an arbitrary subset of $\mathbb{D}$ with fixed size $|S_j| = L < P$, and consider the mean square error between $(\hat{u}_l)_{l=1}^{L}$ and $\left(\left(\mathcal{R}(\Phi)\right)(x_l)\right)_{l=1}^{L}$, where $((x_{l},\hat{u}_{l}))_{l=1}^{L}$ are the elements of $S_j$, that is
\begin{equation}\label{eq:Loss}
	\hbox{Loss}\left(\Phi,(x_l,\hat{u}_l)_{l=1}^{L}\right) = \frac{1}{L} \sum_{l=1}^{L} \left(\left(\mathcal{R}(\Phi)\right)(x_l) - \hat{u}_l \right)^2.
\end{equation}
In each step $j=1,...,N_{Iter}$, the values of weights and biases $(W_{i},b_{i})_{i=1}^{8}$ are updated via the minimization of the loss function \eqref{eq:Loss}. We will do the minimization step with the optimizer \emph{ADAM} implemented in the \emph{Tensorflow} library.

\medskip
We can state therefore the following SGD meta-algorithm in order to obtain an optimal neural network $\Phi^{*}$.

\begin{algorithm}[H]
	\caption{SGD algorithm with batches}\label{alg:SGD}
	\begin{algorithmic}
		\State \textbf{Input:} $N_{iter},M,P \in \N$, $\gamma \in (0,1)$, $L<P$, $L \in \N$
		\State \textbf{Output:} $\Phi^* := (W_i^{*},b_i^{*})_{i=1}^{8}$
		\State Create $\mathbb{D}$ of size $P$ as equal as in \eqref{eq:traindata}
		\State Initialize $\Phi := (W_i,b_i)_{i=1}^{8}$
		\For{$j=1,..,N_{Iter}$}
		\begin{itemize}
		\item Denote $S_j \subset \mathbb{D}$ as a set of $L$ random points of the training set,
		\item Calculate the loss function as in \eqref{eq:Loss} with the data set $S_j$,
		\item Update $(W_{i},b_{i})_{i=1}^{8}$ via the minimization of the loss function using the ADAM optimizer.
		\end{itemize}
		\EndFor
		\For{$i=1,...,8$}
		\State $W_i^{*} \gets W_i$
		\State $b_i^{*} \gets b_i$
		\EndFor
	\end{algorithmic}
\end{algorithm}

{\color{black}\subsubsection{Loss function for radial solutions}

For the cases where the solution is radial, i.e., it only depends on $|x|$, we perform a reinforcement of the loss function, by adding a term that controls the radiality on the realization of the obtained DNN. The loss function we will consider for the cases where the solution is radial is
\begin{equation}\label{eq:loss-radial}
\hbox{Loss}\left(\Phi,(x_l,\hat{u}_l)_{l=1}^{L}\right) = \frac{1}{2L} \sum_{l=1}^{L} \left(\left(\mathcal{R}(\Phi)\right)(x_l) - \hat{u}_l \right)^2 + \left(\left(\mathcal R(\Phi)\right)(x_{l})-\left(\mathcal R(\Phi)\right)(-x_{l})\right)^{2}.
\end{equation}
}

\subsection{Method's error estimates}\label{SSec:3.4}
In the next numerical examples, we will work with PDEs which have explicit solutions {\color{black}$v \in C(\R^d,\R)$}. For the quantification of the following errors, we sample 5000 points $x_k \in D$ or near the boundary $\partial D$. Each point will have the corresponding solution $v(x_k)$.

\medskip

We will compare the optimal neural network $\Phi^{*}$ with the solution of the sampled points in the following two ways:

\begin{enumerate}
	\item Mean Square Error:
	\begin{equation}\label{eq:MSE}
	MSE(\Phi^*,(x_k,v(x_k))_{k=1}^{5000}) = \frac{1}{5000}\sum_{k=1}^{5000} ((\mathcal{R}(\Phi^*))(x_k)-v(x_k))^{2}  .
	\end{equation}
	\item Mean {\color{black}Relative} Error:
	\begin{equation}\label{eq:AE}
	MRE(\Phi^*,(x_k,v(x_k))_{k=1}^{5000}) = \frac{1}{5000}\sum_{k=1}^{5000} \frac{|(\mathcal{R}(\Phi^*))(x_k) - v(x_k)|}{v(x_k)}.
	\end{equation}
\end{enumerate}
This comparison is made in order to study the accuracy of the trained neural network obtained with the SGD algorithm in each example.

\section{Numerical examples}\label{Sec:4}
{\color{black}All the numerical examples will be performed in Python on a 64-bit MacBook Pro M2 (2022) with 8GB of RAM. The codes are available in the author's Github: \url{https://github.com/nvalenzuelaf/DNN-Fractional-Laplacian}}\\

In this paper we provide four different examples of the fractional Dirichlet problem, and they differs in the settings of $f$ and $g$. Each example will be considered over the domain $D = B(0,1) \subset \R^d$, for different values of the dimension $d$.\\

In addition, in each example the points $x_k$ of the training set $\mathbb D$ are chosen such that $\norm{x_k} \leq 1.5$, so we are sampling points on $D$ or outside $D$ but near the boundary $\partial D$.

\subsection{Example 1: Constant source term}\label{SSec:4.1}
Let $d \geq 2$. Consider the following problem 
\[
\centering
\begin{cases}
	(-\Delta)^{\frac{\alpha}2} u(x) &= 2^{\alpha} \frac{\Gamma\left(\frac{\alpha}2 + \frac d2\right)\Gamma\left(\frac{\alpha}2 + 1\right)}{\Gamma\left(\frac d2\right)} \qquad x \in B(0,1),\\
	\hfill u(x) &= 0 \hfill x \notin B(0,1).
\end{cases}
\]
This problem has an explicit {\color{black}continuous }solution $u$ given by \cite{Ex1}.
\begin{equation}\label{eq:sol_ex1}
u(x) = (1-\norm{x}^2)_{+}^{\frac{\alpha}2},
\end{equation}
where $(\cdot)_{+}:= \max\{0,\cdot\}$. {\color{black}The solution \eqref{eq:sol_ex1} is radial, then we use the loss function defined in \eqref{eq:loss-radial} in the SGD algorithm}. For this example we may consider three different values of $d$: $d=2,5,15$. Also, we will set $N_{Iter}=1000$, $P = 2000$, $L = 400$, $M = 100$ and a learning rate $\gamma = 5\times10^{-3}$.

\medskip
First we set $d=2$. In Figure \ref{fig:2D-d2} we compare the realization of the optimal DNN with the solution \eqref{eq:sol_ex1} for $\alpha = 0.1 , 1, 1.9$ in a grid of {\color{black}5000} points between 0 and $\frac{1}{\sqrt{2}}+0.1$. The $x$-axis is the value of both components $x_1 = x_2$. In Figure \ref{fig:3D-d2} we do the comparative in 3D in a grid of {\color{black}$1000\times1000$} points in $[-1,1]\times[-1,1]$. The $x$-axis is the value of $x_1$ and the $y$-axis is the value of $x_2$.

\medskip
In both Figures \ref{fig:2D-d2} and \ref{fig:3D-d2} we see that the optimal neural network fits pretty well the solution \eqref{eq:sol_ex1}. For the case of $\alpha=0.1$ the algorithm takes {\color{black} $22.02$} seconds to obtain the optimal DNN. When $\alpha = 1$ it takes {\color{black}$21.38$} seconds and for the case $\alpha = 1.9$ it takes {\color{black}1 minute and 23.11 seconds}.


\begin{figure}[H]
	\centering
	\begin{subfigure}[b]{0.4\textwidth}
		\centering
		\includegraphics[width=\textwidth]{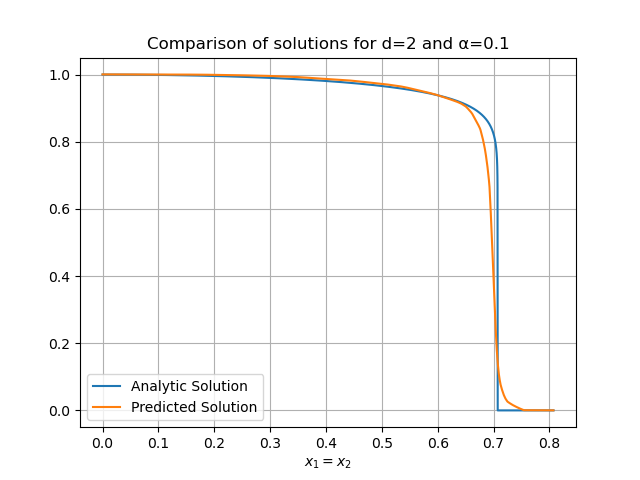}
		\caption{$\alpha = 0.1$}
		\label{fig:d2a01}
	\end{subfigure}
	\begin{subfigure}[b]{0.4\textwidth}
		\centering
		\includegraphics[width=\textwidth]{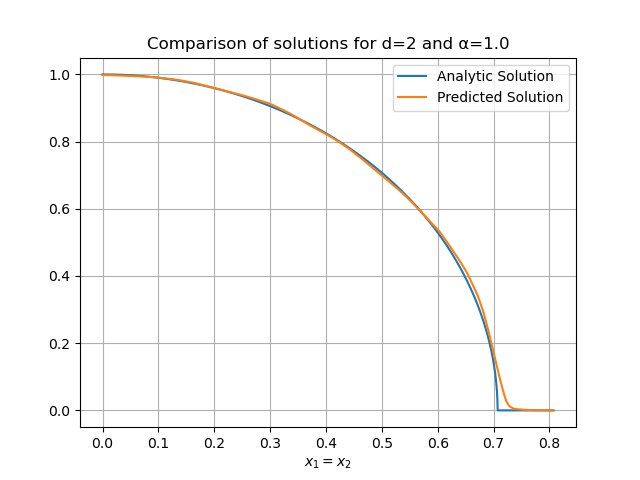}
		\caption{$\alpha = 1$}
		\label{fig:d2a10}
	\end{subfigure}
	\begin{subfigure}[b]{0.4\textwidth}
		\centering
		\includegraphics[width=\textwidth]{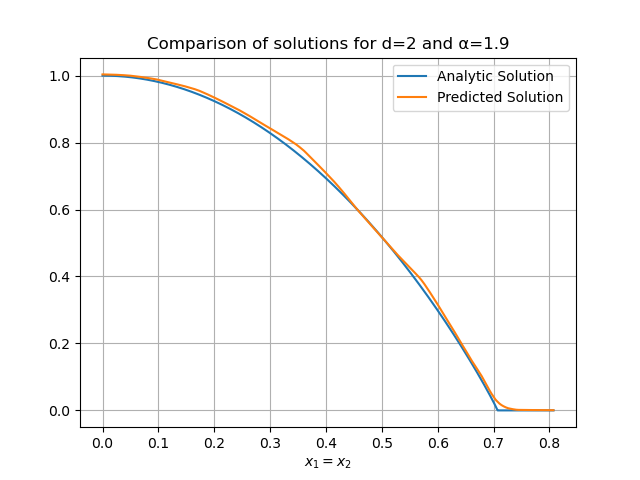}
		\caption{$\alpha = 1.9$}
		\label{fig:d2a19}
	\end{subfigure}
	\caption{Two dimensional comparison of DNN and exact solution when $d=2$.}
\label{fig:2D-d2}
\end{figure}

\begin{figure}[H]
\centering
\begin{subfigure}[b]{0.4\textwidth}
	\centering
	\includegraphics[width=\textwidth]{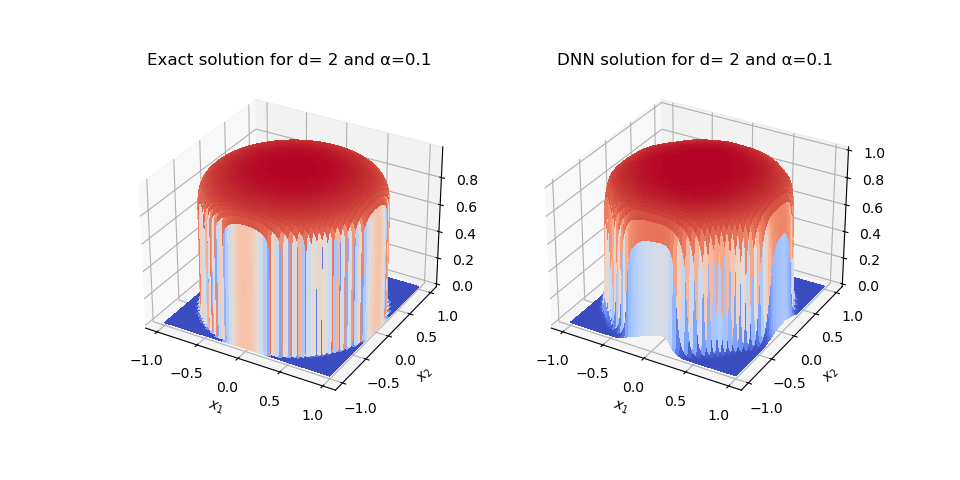}
	\caption{$\alpha = 0.1$}
	\label{fig:3Dd2a01}
\end{subfigure}
\begin{subfigure}[b]{0.4\textwidth}
	\centering
	\includegraphics[width=\textwidth]{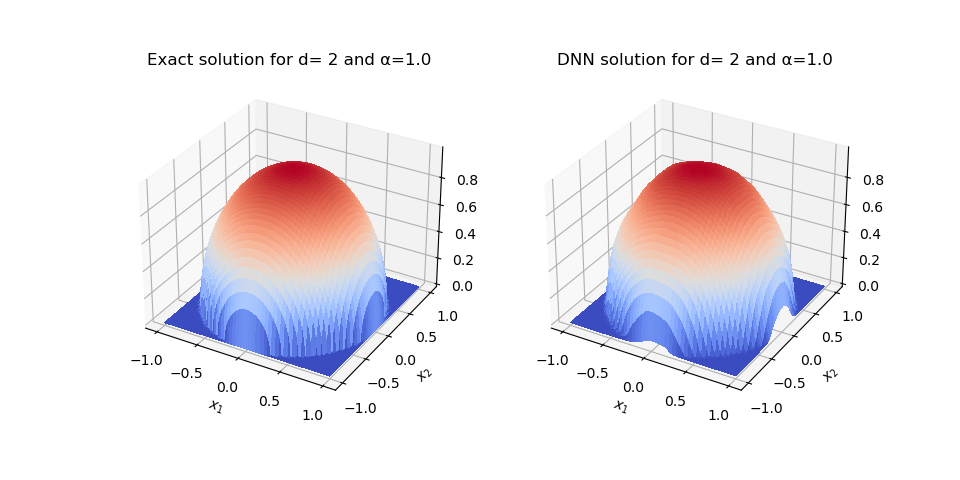}
	\caption{$\alpha = 1$}
	\label{fig:3Dd2a10}
\end{subfigure}
\begin{subfigure}[b]{0.4\textwidth}
	\centering
	\includegraphics[width=\textwidth]{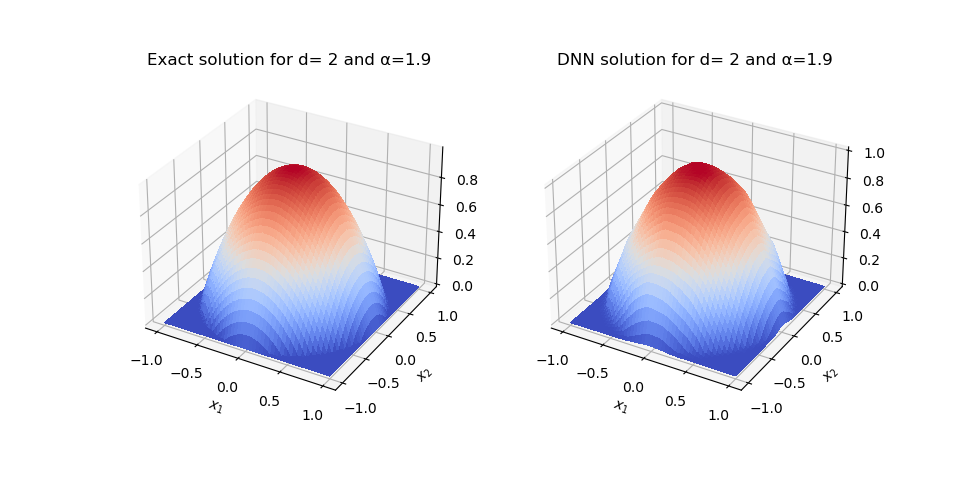}
	\caption{$\alpha = 1.9$}
	\label{fig:3Dd2a19}
\end{subfigure}
\caption{Three dimensional comparison of DNN and exact solution for $d=2$. In each Figure, left: exact solution, right: DNN solution.}
\label{fig:3D-d2}
\end{figure}

{\color{black}In Figure \ref{fig:2D-d2} the most visible difference between the DNN and the solution \eqref{eq:sol_ex1} is near the boundary $\partial B(0,1)$, principally in the cases $\alpha=0.1$ and $\alpha=1.0$. This can also be seen in Figure \ref{fig:3D-d2},  in the points where one component is near zero and the other component is near 1 or -1.}


\begin{figure}[H]
	\centering
	\begin{subfigure}[b]{0.4\textwidth}
		\centering
		\includegraphics[width=\textwidth]{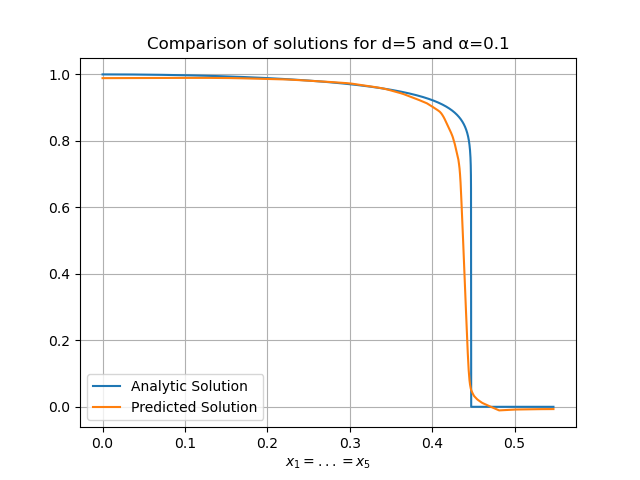}
		\caption{$\alpha = 0.1$}
		\label{fig:d5a01}
	\end{subfigure}
	\begin{subfigure}[b]{0.4\textwidth}
		\centering
		\includegraphics[width=\textwidth]{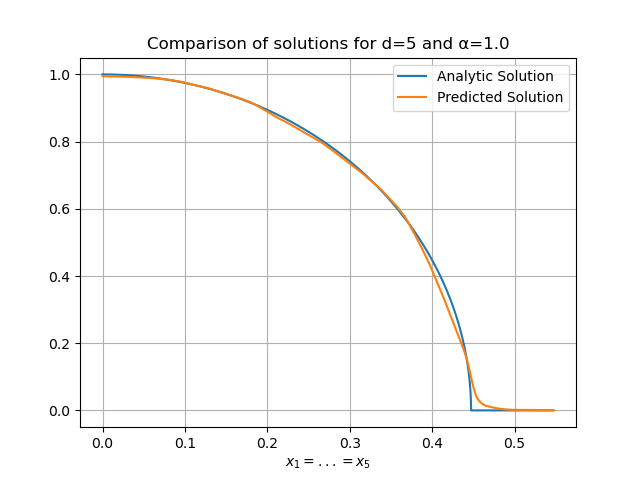}
		\caption{$\alpha = 1$}
		\label{fig:d5a10}
	\end{subfigure}
	\begin{subfigure}[b]{0.4\textwidth}
		\centering
		\includegraphics[width=\textwidth]{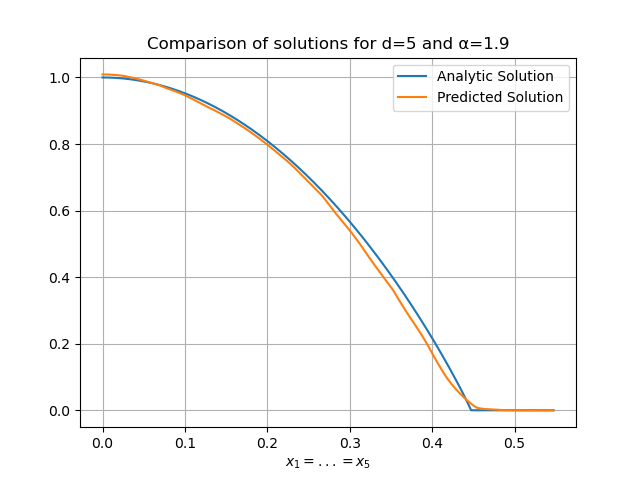}
		\caption{$\alpha = 1.9$}
		\label{fig:d5a19}
	\end{subfigure}
	\caption{Two dimensional comparison of DNN and exact solution when $d=5$.} 
\label{fig:2D-d5}
\end{figure}

\medskip
Now we set $d=5$. Figure \ref{fig:2D-d5} shows the comparative between the optimal DNN and the solution \eqref{eq:sol_ex1} for $\alpha = 0.1, 1, 1.9$ in a grid of {\color{black}5000} points between 0 and $\frac 1{\sqrt 5}+0.1$. The $x$-axis is the value of the components $x_1 = ... = x_5$. Figure \ref{fig:3D-d5} shows the comparative in 3D in a grid of {\color{black}1000$\times$1000} points in $[-1,1]\times[-1,1]$. The $x$-axis is the value of $x_1=...=x_4$ and the $y$-axis is the value of $x_5$.

\begin{figure}[H]
\centering
\begin{subfigure}[b]{0.4\textwidth}
	\centering
	\includegraphics[width=\textwidth]{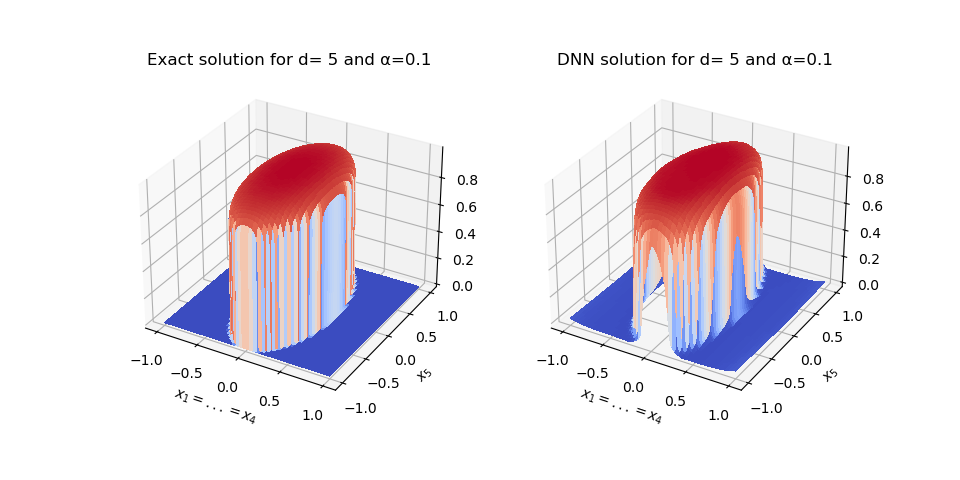}
	\caption{$\alpha = 0.1$}
	\label{fig:3Dd5a01}
\end{subfigure}
\begin{subfigure}[b]{0.4\textwidth}
	\centering
	\includegraphics[width=\textwidth]{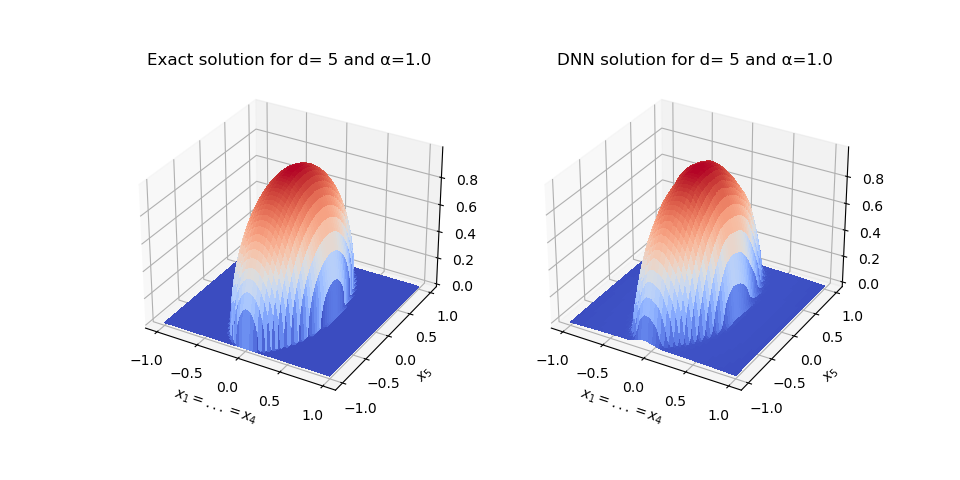}
	\caption{$\alpha = 1$}
	\label{fig:3Dd5a10}
\end{subfigure}
\begin{subfigure}[b]{0.4\textwidth}
	\centering
	\includegraphics[width=\textwidth]{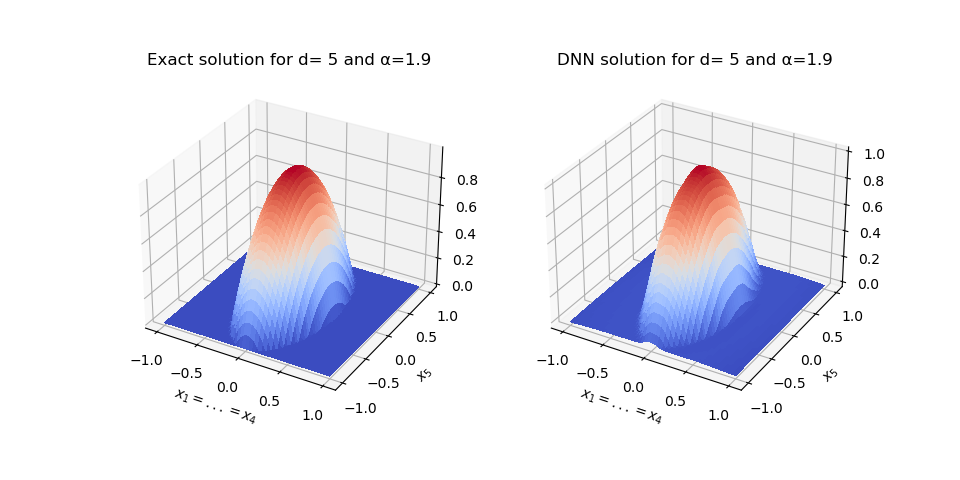}
	\caption{$\alpha = 1.9$}
	\label{fig:3Dd5a19}
\end{subfigure}
\caption{Three dimensional comparison of DNN and exact solution for $d=5$. In each Figure, left: exact solution, right: DNN solution.}
\label{fig:3D-d5}
\end{figure}

\medskip
As we can see in Figure \ref{fig:2D-d5} the approximation via DNN is good. The results {\color{black}have} the most visible difference between the prediction and the solution is near the boundary $\partial B(0,1)$ for the {\color{black}three cases}. In Figure \ref{fig:3D-d5} we see the same behavior that in the case $d=2$, that is, near 0 in the $x$-axis and near $\{-1,1\}$ in the $y$-axis there are a notable difference between the optimal DNN and the solution \eqref{eq:sol_ex1}.


\medskip
In order to improve the the approximation, we {\color{black}perform} a study of the algorithm when we increase the number of iterations $M$ of the Monte Carlo simulation and the size of the training set $P$. In particular, for the same values of $\alpha$ as before, we calculate the elapsed time of the algorithm when $M=10,100,1000,2000$ and $P=10,100,1000,2000$. The results are shown in seconds in Tables \ref{tab:ex1a01}, \ref{tab:ex1a10} and \ref{tab:ex1a19}.

\begin{table}[h]
    \begin{subtable}[h]{0.45\textwidth}
        \centering
        \begin{tabular}{|l|l|l|l|l|}
		\hline
		M\textbackslash{}P & 10    & 100   & 1000   & 2000   \\ \hline
		10                 & 5.44  & 5.62  & 8.35   & 10.49  \\ \hline
		100                & 5.45  & 6.25  & 14.13  & 22.26  \\ \hline
		1000               & 5.94  & 12.46 & 70.90 & 139.47 \\ \hline
		2000               & 6.61 & 17.29 & 140.75 & 274.69 \\ \hline
       \end{tabular}
       \caption{$\alpha = 0.1$}
       \label{tab:ex1a01}
    \end{subtable}
    \hfill
    \begin{subtable}[h]{0.45\textwidth}
        \centering
\begin{tabular}{|l|l|l|l|l|}
	\hline
	M\textbackslash{}P & 10    & 100   & 1000   & 2000    \\ \hline
	10                 & 5.50 & 5.76 & 8.67  & 10.90   \\ \hline
	100                & 5.54  & 6.56 & 15.13  & 23.54   \\ \hline
	1000               & 6.47 & 13.32 & 79.53 & 151.38  \\ \hline
	2000               & 7.00 & 20.19 & 151.30 & 294.83 \\ \hline
\end{tabular}
        \caption{$\alpha = 1.0$}
        \label{tab:ex1a10}
     \end{subtable}
     
     \medskip
     \begin{subtable}[h]{0.45\textwidth}
        \centering
\begin{tabular}{|l|l|l|l|l|}
\hline
M\textbackslash{}P & 10    & 100    & 1000    & 2000     \\ \hline
10                 & 5.60  & 6.72  & 17.86   & 29.27    \\ \hline
100                & 6.09 & 14.40  & 101.61  & 203.21   \\ \hline
1000               & 18.27 & 101.77 & 981.54 & 2012.45  \\ \hline
2000               & 22.54 & 197.30 & 2012.85 & 5556.19 \\ \hline
\end{tabular}
       \caption{$\alpha = 1.9$}
       \label{tab:ex1a19}
    \end{subtable}
     \caption{Example 1: Elapsed time in seconds of the SGD algorithm with $d=5$}
     \label{tab:ex1}
\end{table}

As we can see in the three tables, the setting that takes the longest elapsed time is when $M=2000$ and $P=2000$. This seems reasonable because we need to do a more expensive Monte Carlo simulation for a larger number of points in $\R^5$.

\medskip
Another relationship between elapsed time and PDE parameters can be seen in Tables \ref{tab:ex1a01}, \ref{tab:ex1a10} and \ref{tab:ex1a19}. In particular, for small values of $\alpha$ the algorithm takes less time than when $\alpha$ is larger.

\medskip
Also, from Tables \ref{tab:ex1a01}, \ref{tab:ex1a10} and \ref{tab:ex1a19}, we can deduce that for fixed $M$, the value of the elapsed time when $P$ increases is similar to the value of the elapsed time when $M$ increases and $P$ is fixed.

\medskip
Now, for the same values of $M$ and $P$, and for $\alpha = 0.1,0.5,1,1.5,1.9$ the Figure \ref{fig:MSEP} shows the mean square error calculated as in \eqref{eq:MSE} for $5000$ random points in $D$. On the other hand, Figure \ref{fig:AEP} shows the mean {\color{black}relative} error calculated as in \eqref{eq:AE} for the same $5000$ points in $D$.

\medskip
\begin{figure}[H]
	\centering
	\begin{subfigure}[b]{0.4\textwidth}
		\centering
		\includegraphics[width=\textwidth]{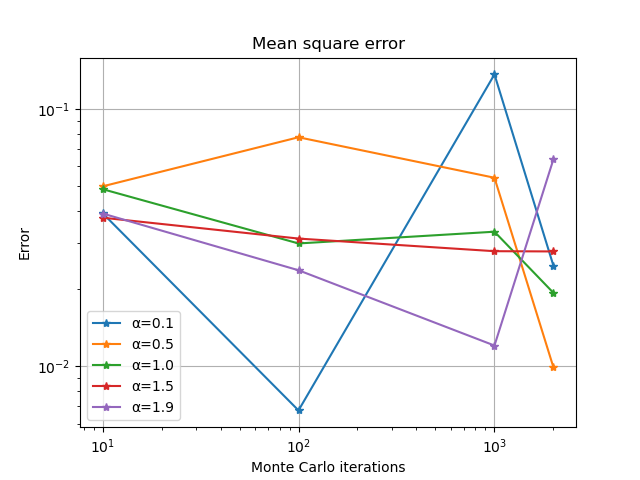}
		\caption{$P = 10$}
		\label{fig:MSEP10d5}
	\end{subfigure}
	\begin{subfigure}[b]{0.4\textwidth}
		\centering
		\includegraphics[width=\textwidth]{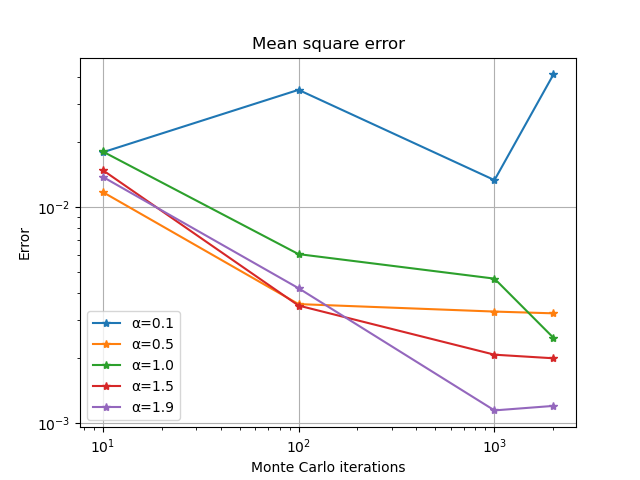}
		\caption{$P = 100$}
		\label{fig:MSEP100d5}
	\end{subfigure}
	\begin{subfigure}[b]{0.4\textwidth}
		\centering
		\includegraphics[width=\textwidth]{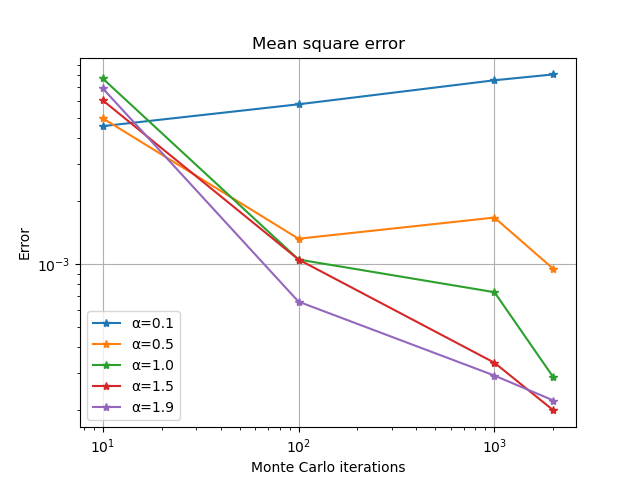}
		\caption{$P = 1000$}
		\label{fig:MSEP1000d5}
	\end{subfigure}
	\centering
	\begin{subfigure}[b]{0.4\textwidth}
		\centering
		\includegraphics[width=\textwidth]{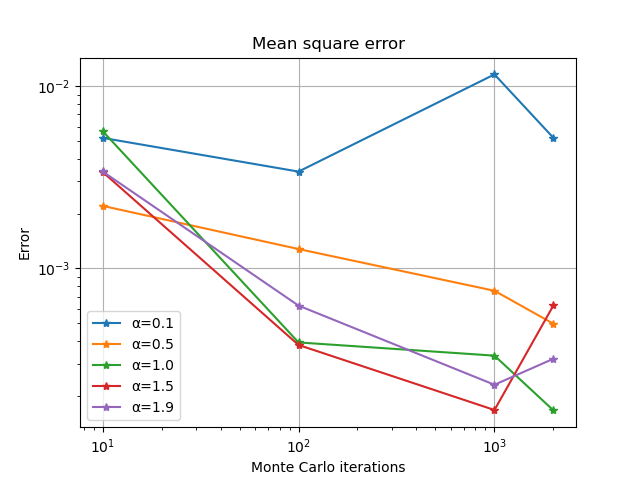}
		\caption{$P = 2000$}
		\label{fig:MSEP2000d5}
	\end{subfigure}
	\caption{Mean square error against number of Monte Carlo iterations $M$ for 5 values of $\alpha$.}
	\label{fig:MSEP}
\end{figure}

\begin{figure}[H]
	\centering
	\begin{subfigure}[b]{0.4\textwidth}
		\centering
		\includegraphics[width=\textwidth]{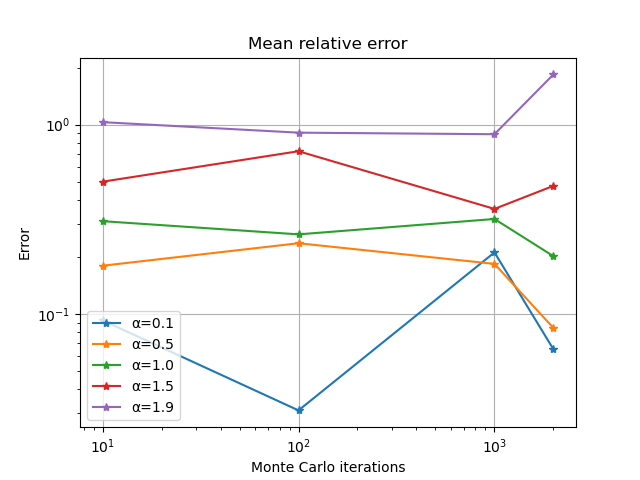}
		\caption{$P = 10$}
		\label{fig:AEP10d5}
	\end{subfigure}
	\begin{subfigure}[b]{0.4\textwidth}
		\centering
		\includegraphics[width=\textwidth]{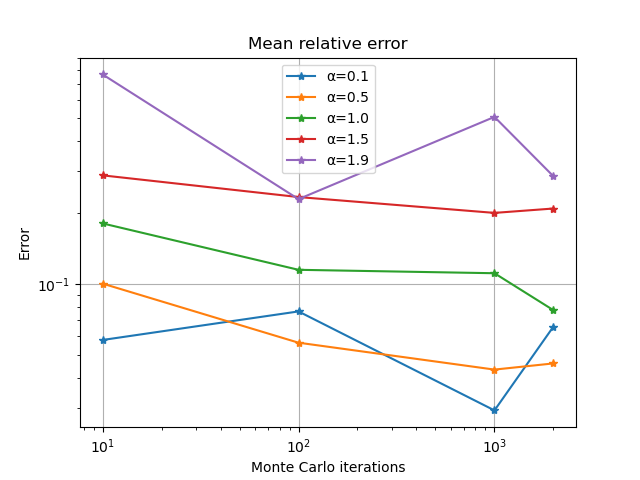}
		\caption{$P = 100$}
		\label{fig:AEP100d5}
	\end{subfigure}
	\hfill
	\begin{subfigure}[b]{0.4\textwidth}
		\centering
		\includegraphics[width=\textwidth]{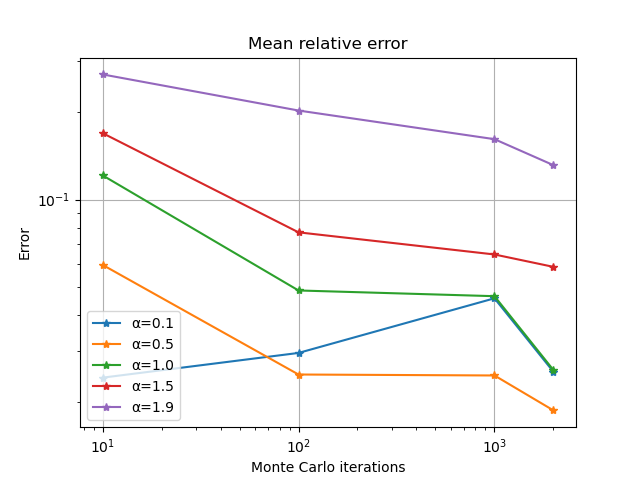}
		\caption{$P = 1000$}
		\label{fig:AEP1000d5}
	\end{subfigure}
	\centering
	\begin{subfigure}[b]{0.4\textwidth}
		\centering
		\includegraphics[width=\textwidth]{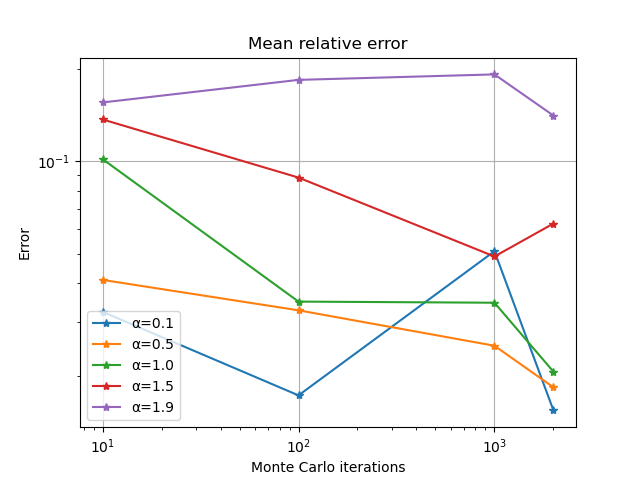}
		\caption{$P = 2000$}
		\label{fig:AEP2000d5}
	\end{subfigure}
	\caption{Mean {\color{black}relative} error against number of Monte Carlo iterations $M$ for 5 values of $\alpha$.}
	\label{fig:AEP}
\end{figure}

{\color{black}\medskip
Notice from Figure \ref{fig:MSEP} that for small values of $\alpha$, one needs big values of $P$. For example, when $\alpha = 0.1$, $P = 1000$ is not enough to have a mean square error of the order of $10^{-4}$, even with $M=2000$. On the other hand, when $\alpha = 1.9$, $P=1000$ gives the order of $10^{-4}$ on the mean square error for $M\geq100$ . Also notice that, despite that the high elapsed time when $P=2000$, in almost all the choices of $M$ and $\alpha$, this setting ensures a mean square error of the order of $10^{-3}$, and the order of $10^{-4}$ if $M$ is greater than $100$.}


{\color{black}\medskip
In Figure \ref{fig:AEP} we see that for the mean relative error is more difficult to obtain a value small order. In particular, only for $P\geq1000$ we can ensure an error lower than the order $10^{-1}$ for almost all the choices of $M$ and $\alpha$. The excluding case is when $\alpha = 1.9$, where the errors are greater than $10^{-1}$, but they are still in that order.}


\medskip
Finally, we set $d=15$ {\color{black}and we change the value of $\gamma$ to $ 5 \times 10^{-4}$.} In Figure \ref{fig:2D-d15} we compare the realization of the optimal DNN with the solution \eqref{eq:sol_ex1} for $\alpha=0.1,1,1.9$ in a grid of {\color{black}5000} points between 0 and $\frac{1}{\sqrt{15}}+0.1$. The $x$-axis is the value of the components $x_1 = ... = x_{15}$. In Figure \ref{fig:3D-d15} we do the comparative in 3D in a grid of {\color{black}1000$\times$1000} points in $[-1,1]\times[-1,1]$. The $x$-axis is the value of $x_1=...=x_{14}$ and the $y$-axis is the value of $x_{15}$.

\begin{figure}[H]
	\centering
	\begin{subfigure}[b]{0.4\textwidth}
		\centering
		\includegraphics[width=\textwidth]{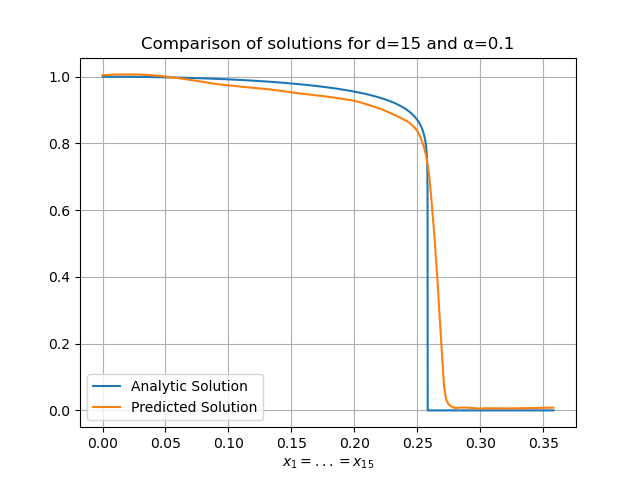}
		\caption{$\alpha = 0.1$}
		\label{fig:d15a01}
	\end{subfigure}
	\begin{subfigure}[b]{0.4\textwidth}
		\centering
		\includegraphics[width=\textwidth]{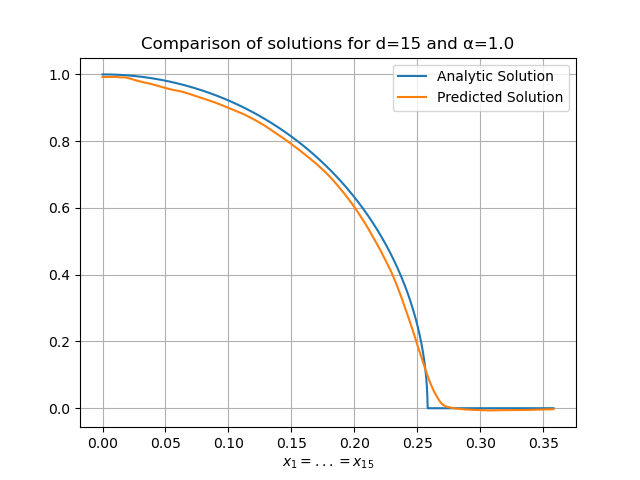}
		\caption{$\alpha = 1$}
		\label{fig:d15a10}
	\end{subfigure}
	\begin{subfigure}[b]{0.4\textwidth}
		\centering
		\includegraphics[width=\textwidth]{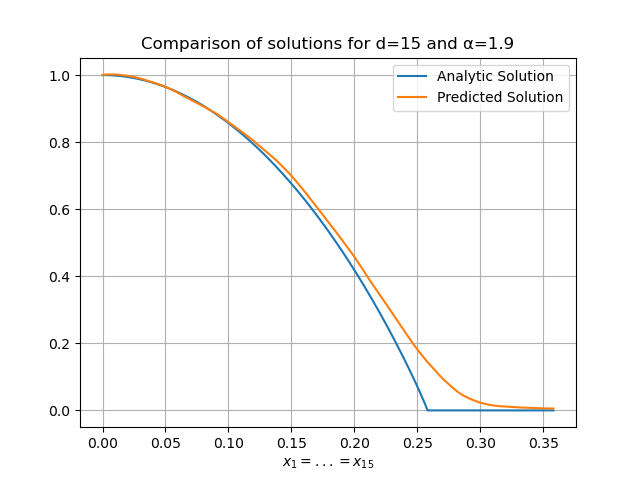}
		\caption{$\alpha = 1.9$}
		\label{fig:d15a19}
	\end{subfigure}
	\caption{Two dimensional comparison of DNN and exact solution when $d=15$.} 
\label{fig:2D-d15}
\end{figure}

\begin{figure}[H]
\centering
\begin{subfigure}[b]{0.4\textwidth}
	\centering
	\includegraphics[width=\textwidth]{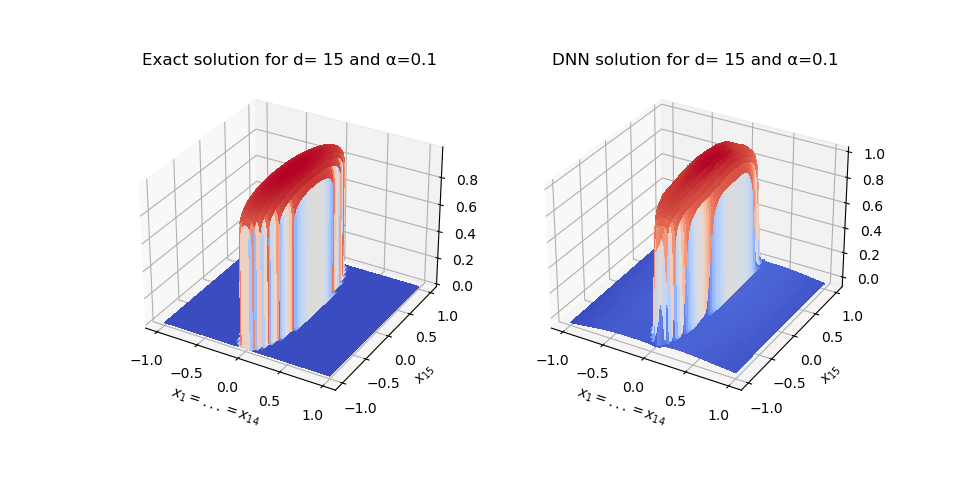}
	\caption{$\alpha = 0.1$}
	\label{fig:3Dd15a01}
\end{subfigure}
\begin{subfigure}[b]{0.4\textwidth}
	\centering
	\includegraphics[width=\textwidth]{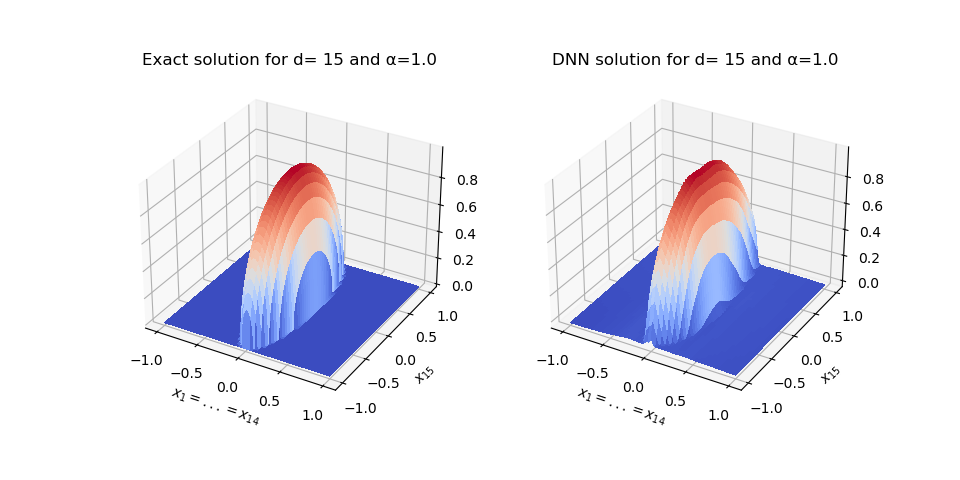}
	\caption{$\alpha = 1$}
	\label{fig:3Dd15a10}
\end{subfigure}
\begin{subfigure}[b]{0.4\textwidth}
	\centering
	\includegraphics[width=\textwidth]{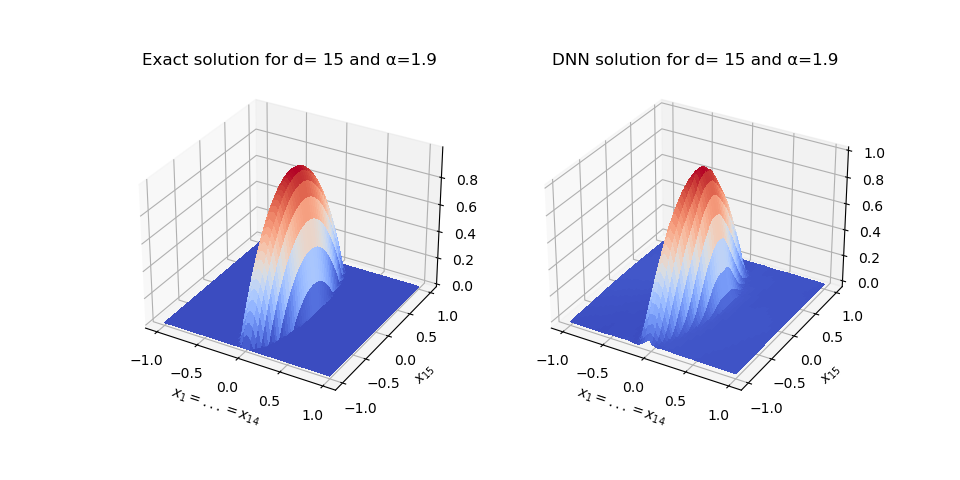}
	\caption{$\alpha = 1.9$}
	\label{fig:3Dd15a19}
\end{subfigure}
\caption{Three dimensional comparison of DNN and exact solution for $d=15$. In each Figure, left: exact solution, right: DNN solution.}
\label{fig:3D-d15}
\end{figure}

{\color{black}\medskip
Notice from Figure \ref{fig:2D-d15} that the behavior is a quite similar that when $d=5$. In the sense that for all the values of $\alpha$, the optimal DNN is not a good prediction of the solution near the boundary $\partial B(0,1)$, specially when $\alpha=0.1$ and $\alpha=1.9$. Despite the not-so-good approximation mentioned above, in Figure \ref{fig:3D-d15} we see that the optimal DNN still preserves the {\color{black} solution's form.}}


{\color{black}\subsubsection{Comparison of loss functions} Recall from \eqref{eq:sol_ex1} the radiality of the solution in this example. As said before, in previous figures we perform the SGD algorithm with the loss function defined in \eqref{eq:loss-radial}. In this section we compare the optimal DNNs of the SGD algorithm for the two different losses described in Section \ref{Sec:4}. The comparison will be made for $\alpha=0.1,1,1.9$, $d=15$, $\gamma=5\times10^{-4}$, $P=2000$, $L=400$ and $M=100$. \\

\medskip
As we can see in Figure \ref{fig:2D-d15-radial}, the performance of the SGD algorithm is better when we are using the loss function \eqref{eq:loss-radial} instead of function \eqref{eq:Loss}. Despite there are not so much difference between the errors obtained with both loss functions, in the visualization on Figure \ref{fig:2D-d15-radial} is remarkable the use of the loss function with radial term. 

\begin{figure}[H]
	\centering
	\begin{subfigure}[b]{0.4\textwidth}
		\centering
		\includegraphics[width=\textwidth]{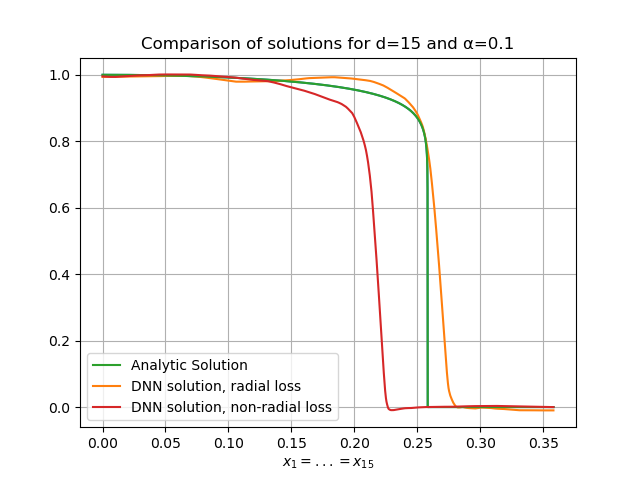}
		\caption{$\alpha = 0.1$}
		\label{fig:d15a01rs}
	\end{subfigure}
	\begin{subfigure}[b]{0.4\textwidth}
		\centering
		\includegraphics[width=\textwidth]{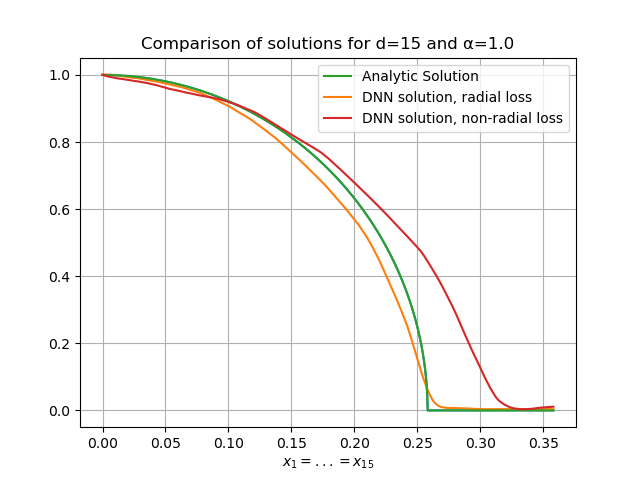}
		\caption{$\alpha = 1$}
		\label{fig:d15a10rs}
	\end{subfigure}
	\begin{subfigure}[b]{0.4\textwidth}
		\centering
		\includegraphics[width=\textwidth]{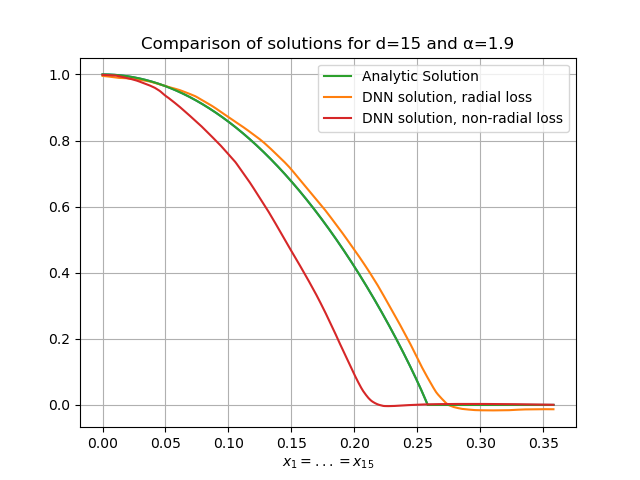}
		\caption{$\alpha = 1.9$}
		\label{fig:d15a19rs}
	\end{subfigure}
	\caption{Two dimensional comparison of two optimal DNN solutions with different loss functions, for $d=15$.} 
\label{fig:2D-d15-radial}
\end{figure}
}

\subsection{Example 2: Non constant source term}

Let $d \geq 2$. Consider the following problem
\begin{equation}\label{eq:sol_ex2}
\centering
\begin{cases}
	(-\Delta)^{\frac{\alpha}2} u(x) &= b_{\alpha,d}\left(1-\left(1+\frac{\alpha}d\right)|x|^2\right) \qquad x \in B(0,1),\\
	\hfill u(x) &= 0 \hfill x \notin B(0,1),
\end{cases}
\end{equation}
where 
\[
b_{\alpha,d} =  2^{\alpha} \frac{\Gamma\left(\frac{\alpha}2 + \frac d2\right)\Gamma\left(\frac{\alpha}2 + 2\right)}{\Gamma\left(\frac d2\right)}.
\]
This problem has an explicit {\color{black}continuous }solution $u$ given by \cite{Ex2}.
\[
u(x) = (1-\norm{x}^2)_{+}^{1+\frac{\alpha}2}.
\]
For this example, {\color{black}is also radial, then the loss function \eqref{eq:loss-radial} will be used}. We consider $d=2$ and $d=5$. For both dimensions, we set $M=500$, $P=1000$, $L=200$, $\gamma = 5 \times 10^{-3}$ and $N_{Iter} = 1000$. 

\medskip
First we set $d=2$. In Figure \ref{fig:ex3-2D-d2} we compare the realization of the optimal DNN with the solution \eqref{eq:sol_ex2} for $\alpha = 0.5,1.5,1.9$ in a grid of {\color{black}5000} points between 0 and $\frac 1{\sqrt 2} + 0.1$. Figure \ref{fig:ex3-3D-d2} we do the comparative in 3D in a grid of {\color{black}1000$\times$1000} points in $[-1,1]\times[-1,1]$.

\begin{figure}[H]
	\centering
	\begin{subfigure}[b]{0.4\textwidth}
		\centering
		\includegraphics[width=\textwidth]{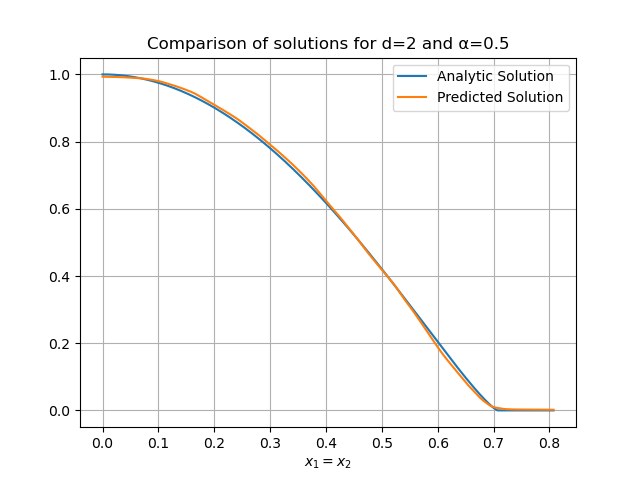}
		\caption{$\alpha = 0.5$}
		\label{fig:ex3d2a05}
	\end{subfigure}
	\begin{subfigure}[b]{0.4\textwidth}
		\centering
		\includegraphics[width=\textwidth]{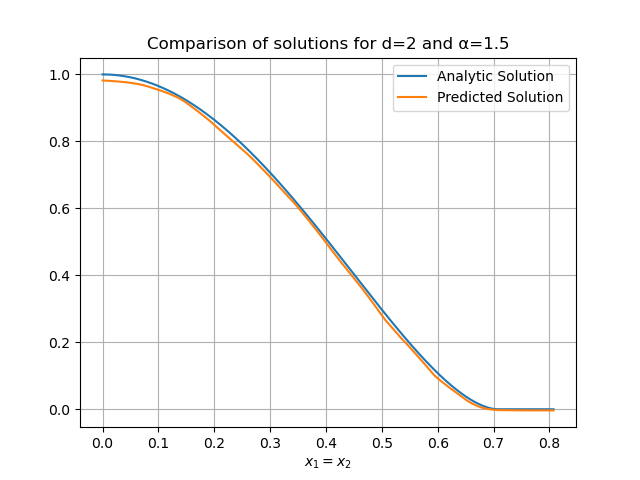}
		\caption{$\alpha = 1.5$}
		\label{fig:ex3d2a15}
	\end{subfigure}
	\begin{subfigure}[b]{0.4\textwidth}
		\centering
		\includegraphics[width=\textwidth]{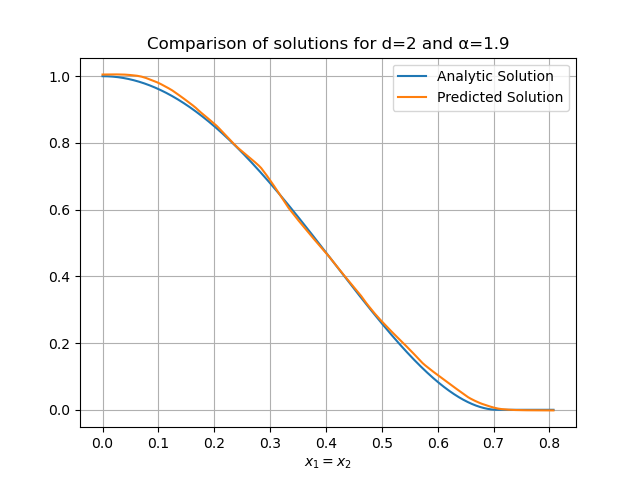}
		\caption{$\alpha = 1.9$}
		\label{fig:ex3d2a19}
	\end{subfigure}
	\caption{Two dimensional comparison of DNN and exact solution when $d=2$.}
\label{fig:ex3-2D-d2}
\end{figure}

\begin{figure}[H]
\centering
\begin{subfigure}[b]{0.4\textwidth}
	\centering
	\includegraphics[width=\textwidth]{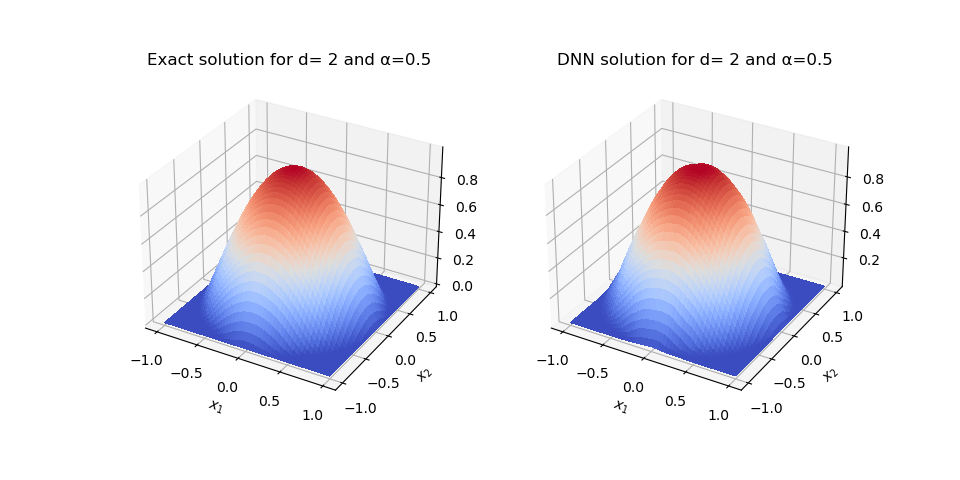}
	\caption{$\alpha = 0.5$}
	\label{fig:ex33Dd2a05}
\end{subfigure}
\begin{subfigure}[b]{0.4\textwidth}
	\centering
	\includegraphics[width=\textwidth]{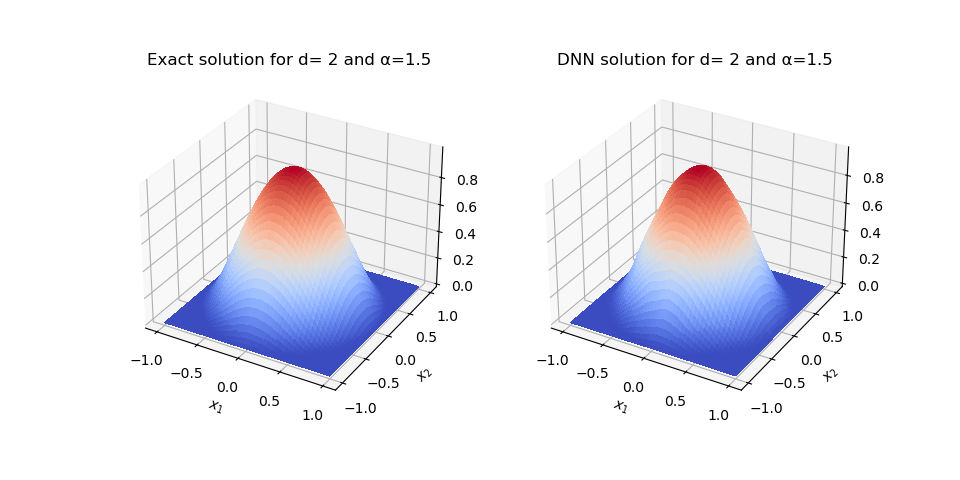}
	\caption{$\alpha = 1.5$}
	\label{fig:ex33Dd2a15}
\end{subfigure}
\begin{subfigure}[b]{0.4\textwidth}
	\centering
	\includegraphics[width=\textwidth]{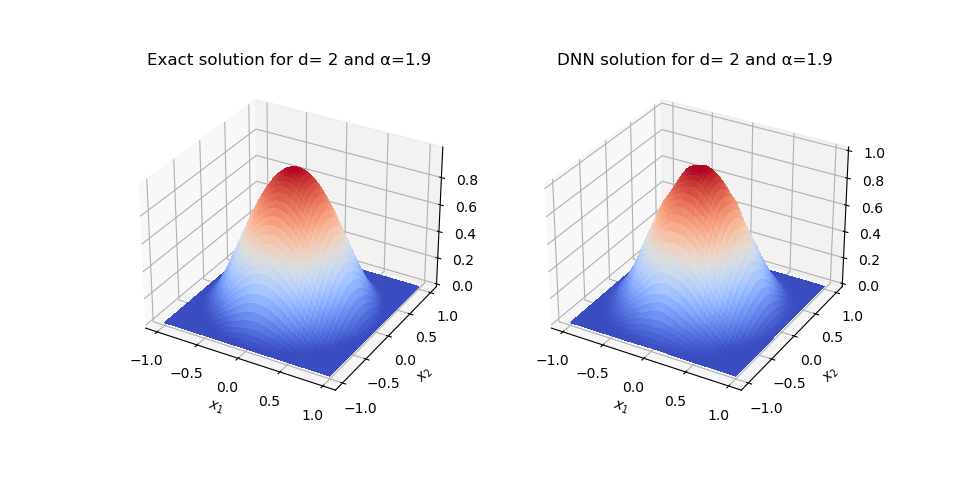}
	\caption{$\alpha = 1.9$}
	\label{fig:ex33Dd2a19}
\end{subfigure}
\caption{Three dimensional comparison of DNN and exact solution for $d=2$. In each Figure, left: exact solution, right: DNN solution.}
\label{fig:ex3-3D-d2}
\end{figure}

{\color{black}\medskip
As in Example 1, the optimal DNN fits pretty well the solution \eqref{eq:sol_ex2} when $d=2$, for all studied cases of $\alpha$, but unlike in previous example, in this case there are no problems near the boundary and there are not major visible differences between the optimal DNN and the solution \eqref{eq:sol_ex2}. On the other hand, Figure \ref{fig:ex3-3D-d2} shows that the optimal DNN has the same order and the same form of the analytical solution.}

\medskip
For $d=2$ we compare the elapsed time for different choices of $M$ and $P$. This was made for $\alpha = 0.5$ (Table \ref{tab:ex3a05}), $\alpha = 1.5$ (Table \ref{tab:ex3a15}) and for $\alpha = 1.9$ (Table \ref{tab:ex3a19}). In those tables we can see that the elapsed time increases when we increase the value of $M$ and $P$. For fixed values of $M$ and $P$ the elapsed time also increases when we increases the value of $\alpha$, but the major gap is for larger values of $M$ and $P$.

\begin{table}[h]
    \begin{subtable}[h]{0.45\textwidth}
        \centering
	\begin{tabular}{|l|l|l|l|l|}
		\hline
		M\textbackslash{}P & 10    & 100   & 1000   & 2000   \\ \hline
		10                 & 3.75  & 3.77  & 5.40   & 6.66  \\ \hline
		100                & 3.71  & 4.35  & 10.39  & 16.61  \\ \hline
		1000               & 4.47  & 9.28 & 60.81 & 116.09 \\ \hline
		2000               & 4.71 & 14.99 & 120.82 & 227.46 \\ \hline
	\end{tabular}
       \caption{$\alpha = 0.5$.}
       \label{tab:ex3a05}
    \end{subtable}
    \hfill
    \begin{subtable}[h]{0.45\textwidth}
        \centering
\begin{tabular}{|l|l|l|l|l|}
	\hline
	M\textbackslash{}P & 10     & 100    & 1000    & 2000    \\ \hline
	10                 & 3.67   & 3.80   & 5.83   & 7.66   \\ \hline
	100                & 3.77   & 4.86  & 15.66   & 26.71   \\ \hline
	1000               & 5.02  & 14.30  &  113.68 & 219.76  \\ \hline
	2000               & 4.85  & 25.93  & 223.80 & 431.75 \\ \hline
\end{tabular}
        \caption{$\alpha = 1.5$.}
        \label{tab:ex3a15}
     \end{subtable}
     
     \medskip
     \begin{subtable}[h]{0.45\textwidth}
        \centering
\begin{tabular}{|l|l|l|l|l|}
\hline
M\textbackslash{}P & 10     & 100      & 1000       & 2000      \\ \hline
10                 & 3.67   & 4.08     & 8.53      & 12.94     \\ \hline
100                & 4.04   & 6.95    & 42.65     & 80.22    \\ \hline
1000               & 7.36  & 40.62   & 384.95    & 760.63   \\ \hline
2000               & 13.30  & 77.12   & 755.87    & 1457.21   \\ \hline
\end{tabular}
       \caption{$\alpha = 1.9$.}
       \label{tab:ex3a19}
    \end{subtable}
     \caption{Example 2: Elapsed time in seconds of the SGD algorithm with $d=2$.}
     \label{tab:ex3}
\end{table}

\medskip
Finally we compare the mean square error and the mean {\color{black}relative} error in a grid of the Monte Carlo iterations $M$, for each $\alpha \in \{0.1,0.5,1,1.5,1.9\}$ and $P = 10,100,1000,2000$. The results are summarized in Figures \ref{fig:ex3-MSEP} and \ref{fig:ex3-AEP}.

\begin{figure}[H]
	\centering
	\begin{subfigure}[b]{0.4\textwidth}
		\centering
		\includegraphics[width=\textwidth]{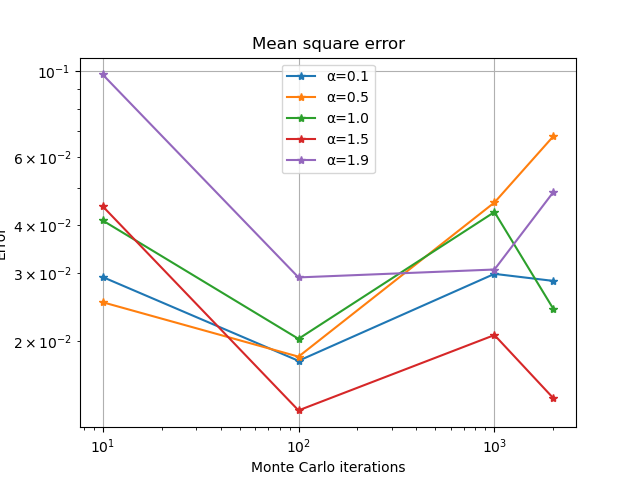}
		\caption{$P = 10$}
		\label{fig:ex3-MSEP10d2}
	\end{subfigure}
	\begin{subfigure}[b]{0.4\textwidth}
		\centering
		\includegraphics[width=\textwidth]{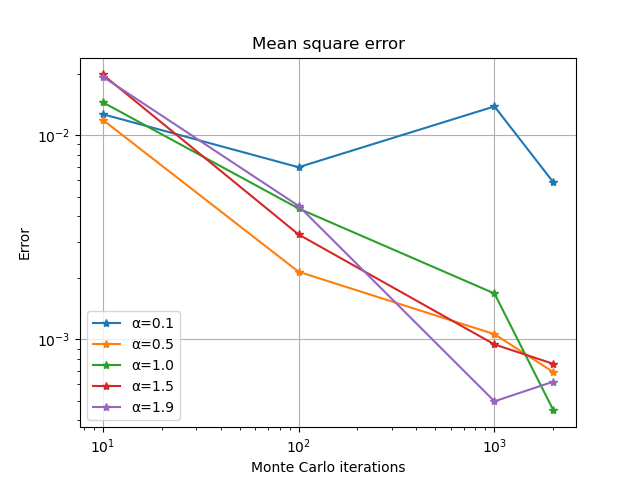}
		\caption{$P = 100$}
		\label{fig:ex3-MSEP100d2}
	\end{subfigure}
	\begin{subfigure}[b]{0.4\textwidth}
		\centering
		\includegraphics[width=\textwidth]{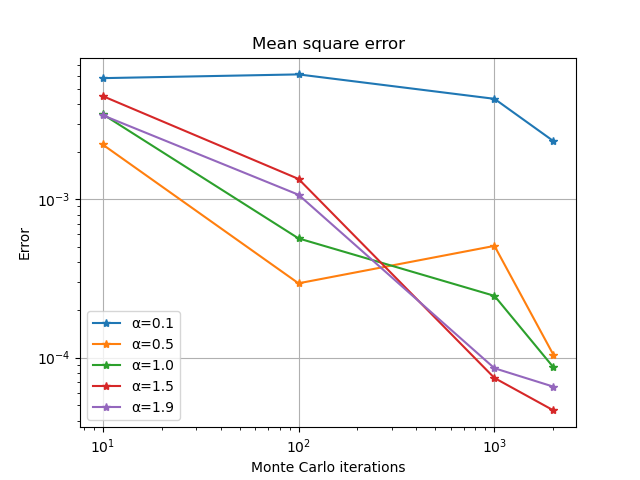}
		\caption{$P = 1000$}
		\label{fig:ex3-MSEP1000d2}
	\end{subfigure}
	\centering
	\begin{subfigure}[b]{0.4\textwidth}
		\centering
		\includegraphics[width=\textwidth]{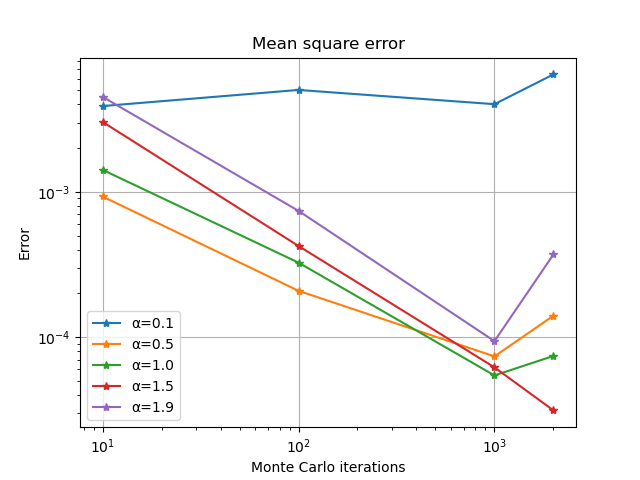}
		\caption{$P = 2000$}
		\label{fig:ex3-MSEP2000d2}
	\end{subfigure}
	\caption{Mean square error against number of Monte Carlo iterations $M$ for 5 values of $\alpha$.}
	\label{fig:ex3-MSEP}
\end{figure}

\medskip
\begin{figure}[H]
	\centering
	\begin{subfigure}[b]{0.4\textwidth}
		\centering
		\includegraphics[width=\textwidth]{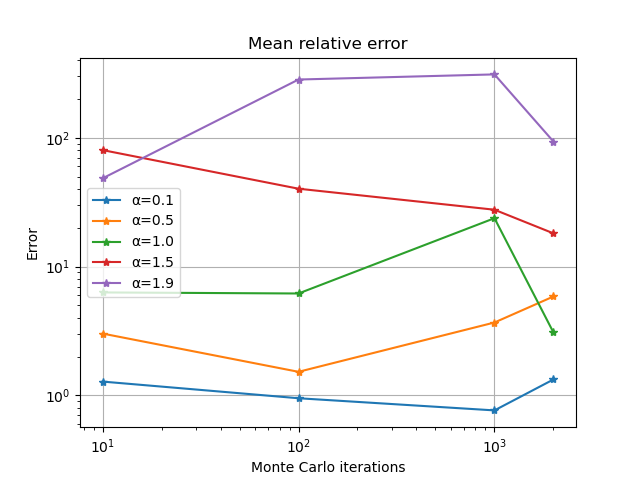}
		\caption{$P = 10$}
		\label{fig:ex3-AEP10d2}
	\end{subfigure}
	\begin{subfigure}[b]{0.4\textwidth}
		\centering
		\includegraphics[width=\textwidth]{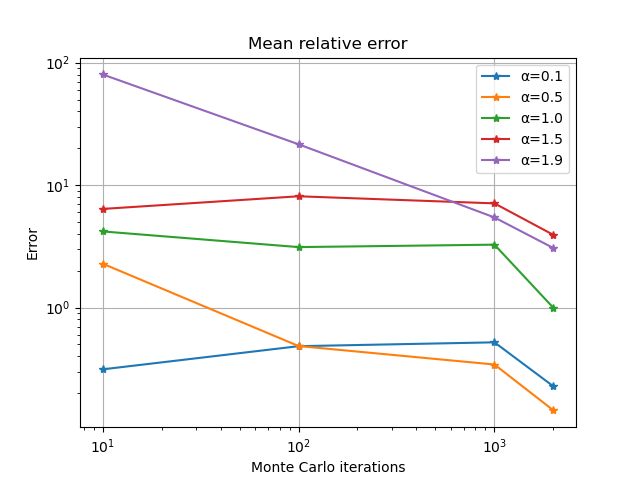}
		\caption{$P = 100$}
		\label{fig:ex3-AEP100d2}
	\end{subfigure}
	\hfill
	\begin{subfigure}[b]{0.4\textwidth}
		\centering
		\includegraphics[width=\textwidth]{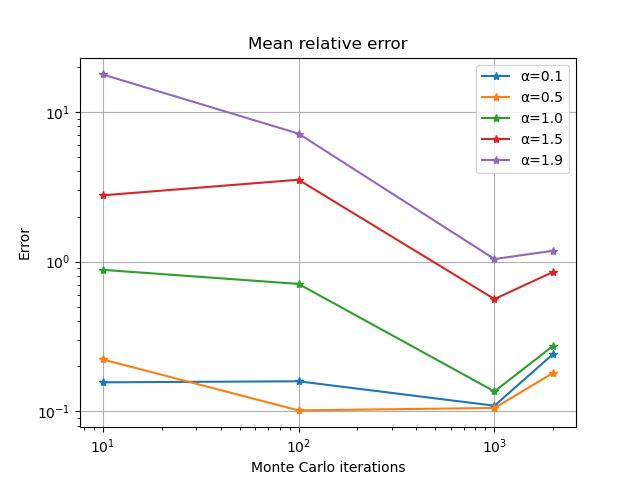}
		\caption{$P = 1000$}
		\label{fig:ex3-AEP1000d2}
	\end{subfigure}
	\centering
	\begin{subfigure}[b]{0.4\textwidth}
		\centering
		\includegraphics[width=\textwidth]{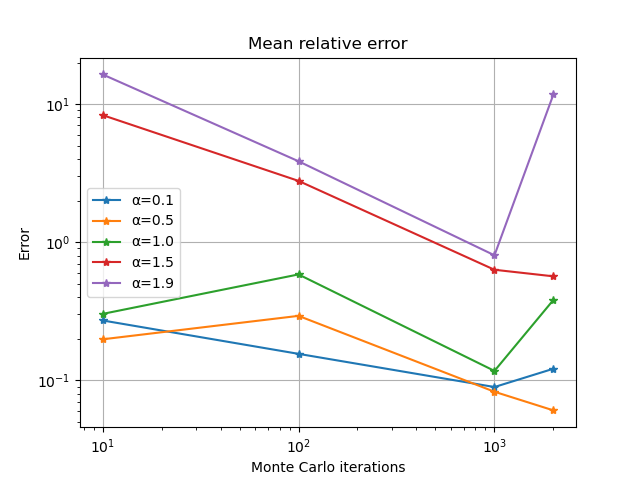}
		\caption{$P = 2000$}
		\label{fig:ex3-AEP2000d2}
	\end{subfigure}
	\caption{Mean {\color{black}relative} error against number of Monte Carlo iterations $M$ for 5 values of $\alpha$.}
	\label{fig:ex3-AEP}
\end{figure}

{\color{black}\medskip
Figure \ref{fig:ex3-MSEP} shows that the mean square error decreases when we increases $P$ and $M$, achieving the order $10^{-4}$  for the cases $P=1000$ and $P=2000$ when $M\geq 100$ and for any value of $\alpha$. In Figure \ref{fig:ex3-AEP} we see a pretty similar behavior of the mean relative error than the MSE, but with different order:  for $P=1000$ and $P=2000$, we have an error of order $10^{-1}$ only for small values of $\alpha$, $\alpha\leq 1$, and for $\alpha >1$, only in the case $M=1000$. For the other cases, the MRE is greater than 1. In other words, the mean square error is small for all the studied values for $\alpha$, but only for small values of $\alpha$ we reach a small value on the mean relative error.
}

\medskip
Now we set $d=5$. We will compare the optimal DNN and the solution \eqref{eq:sol_ex2} in 2D (Figure \ref{fig:ex3-2D-d5}) and in 3D (Figure \ref{fig:ex3-3D-d5}) for $\alpha=0.5,1.5,1.9$. This will be done in the same way as in Example 1. In both Figures we can see a good approximation of the solution \eqref{eq:sol_ex2} with the optimal DNN obtained with the SGD algorithm. In the three settings of $\alpha$ the most visible difference of the approximation is near the center of the domain, i.e, near 0, {\color{black} and near the boundary $\partial B(0,1)$}, specially in the cases $\alpha = 1.5$ and $\alpha = 1.9$. {\color{black}Moreover, Figure \ref{fig:ex3-3D-d5} shows that the form of the optimal DNN is the same of the solution \eqref{eq:sol_ex2}.}


\begin{figure}[H]
	\centering
	\begin{subfigure}[b]{0.4\textwidth}
		\centering
		\includegraphics[width=\textwidth]{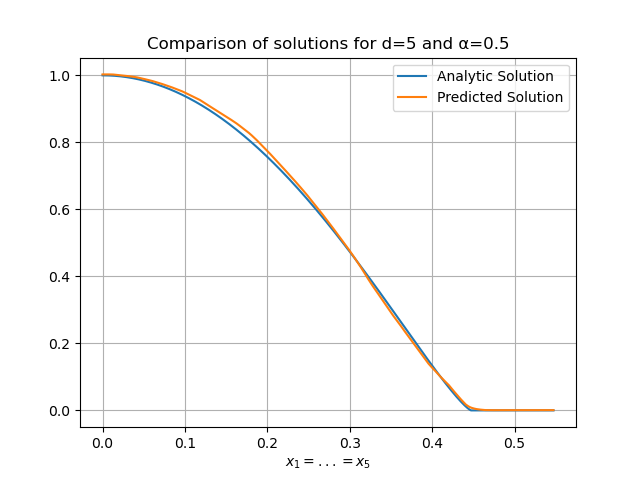}
		\caption{$\alpha = 0.5$}
		\label{fig:ex3d5a05}
	\end{subfigure}
	\begin{subfigure}[b]{0.4\textwidth}
		\centering
		\includegraphics[width=\textwidth]{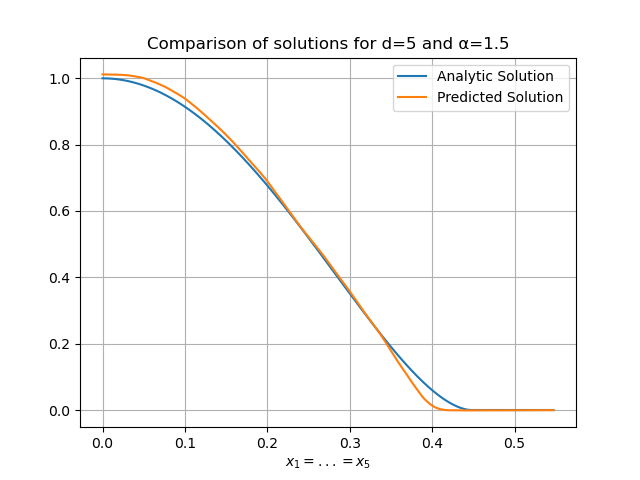}
		\caption{$\alpha = 1.5$}
		\label{fig:ex3d5a15}
	\end{subfigure}
	\begin{subfigure}[b]{0.4\textwidth}
		\centering
		\includegraphics[width=\textwidth]{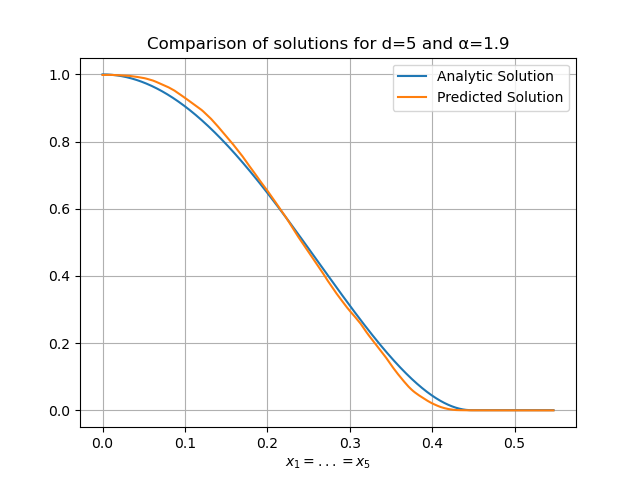}
		\caption{$\alpha = 1.9$}
		\label{fig:ex3d5a19}
	\end{subfigure}
	\caption{Two dimensional comparison of DNN and exact solution when $d=5$.}
\label{fig:ex3-2D-d5}
\end{figure} 

\begin{figure}[H]
\centering
\begin{subfigure}[b]{0.4\textwidth}
	\centering
	\includegraphics[width=\textwidth]{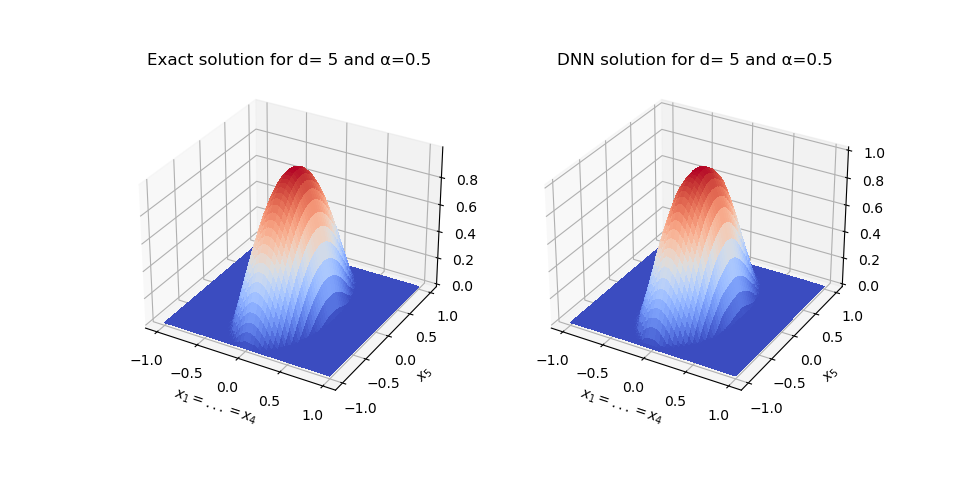}
	\caption{$\alpha = 0.5$}
	\label{fig:ex33Dd5a05}
\end{subfigure}
\begin{subfigure}[b]{0.4\textwidth}
	\centering
	\includegraphics[width=\textwidth]{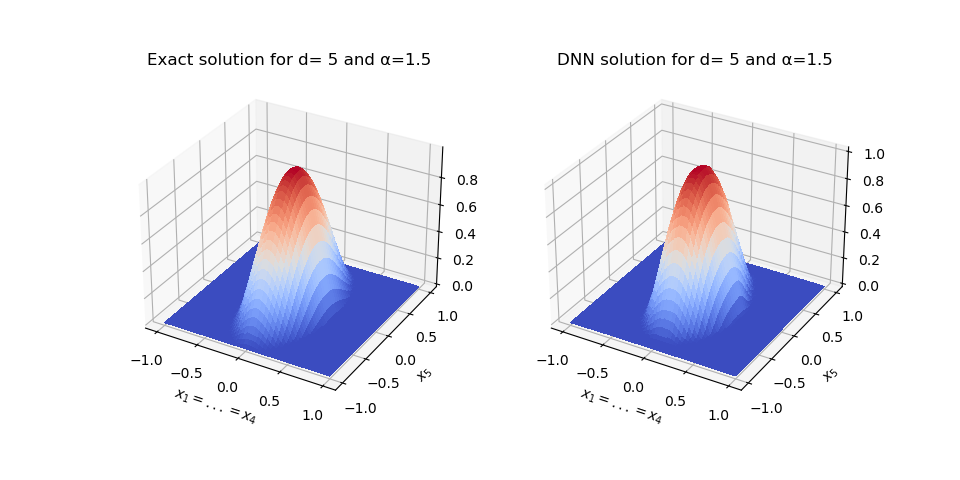}
	\caption{$\alpha = 1.5$}
	\label{fig:ex33Dd5a10}
\end{subfigure}
\begin{subfigure}[b]{0.4\textwidth}
	\centering
	\includegraphics[width=\textwidth]{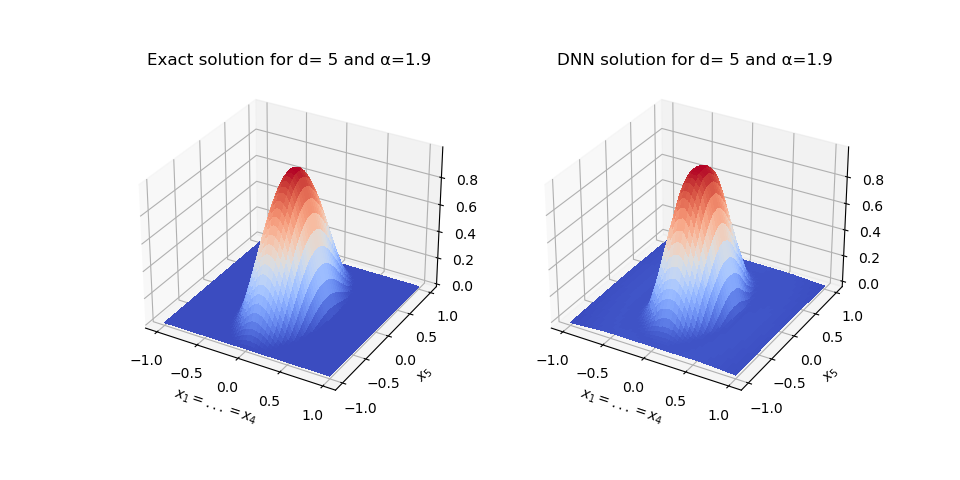}
	\caption{$\alpha = 1.9$}
	\label{fig:ex33Dd5a19}
\end{subfigure}
\caption{Three dimensional comparison of DNN and exact solution for $d=5$. In each Figure, left: exact solution, right: DNN solution.}
\label{fig:ex3-3D-d5}
\end{figure}

\subsection{Example 3: Non zero boundary term}

 Let $d \geq 2$ and recall that $D = B(0,1)$. For $y \in D^c$ fixed let the boundary condition $g$ {\color{black}be} a translation of the fundamental solution of the fractional Laplacian, that is,
\[
g(x) = \frac{\Gamma\left(\frac d2 - \frac{\alpha}2\right)}{2^{\alpha} \pi^{\frac d2}\Gamma\left(\frac{\alpha}2\right)} |x-y|^{\alpha - d}.
\]
Now consider the problem
\begin{equation}\label{eq:ex3}
\centering
\begin{cases}
	(-\Delta)^{\frac{\alpha}2} u(x) &= 0 \hfill x \in B(0,1),\\
	\hfill u(x) &= g(x) \qquad x \notin B(0,1).
\end{cases}
\end{equation}
From the choice of $y$ outside $B(0,1)$, one obtains that {\color{black}one} exact solution for previous problem is $u(x) = g(x)$ for every $x \in \R^d$. In this particular case, we consider the vector $y = (2,0,...,0) \in \R^d$.

\medskip
For this example, we will set the values $M=300$, $P=1000$, $L=200$, $\gamma = 5 \times 10^{-3}$ and $N_{Iter}=1000$. First we study the algorithm in dimension 2. Figure \ref{fig:ex4-2D-d2} shows the comparison of the optimal DNN and the solution of problem \eqref{eq:ex3} for $\alpha = 0.5,1.5,1.9$ in a grid made in the same way as in previous examples. Figure \ref{fig:ex4-3D-d2} shows the comparison in 3D.

\medskip
From Figure \ref{fig:ex4-2D-d2} we can see a poorly approximation of the optimal DNN when $\alpha = 0.5$. The approximation is improved by increasing $\alpha$. By de order of the solution of \eqref{eq:ex3}, we see that the optimal DNN still preserves the form of the analytic solution in the square $[-1,1]\times[-1,1]$.

\begin{figure}[H]
	\centering
	\begin{subfigure}[b]{0.4\textwidth}
		\centering
		\includegraphics[width=\textwidth]{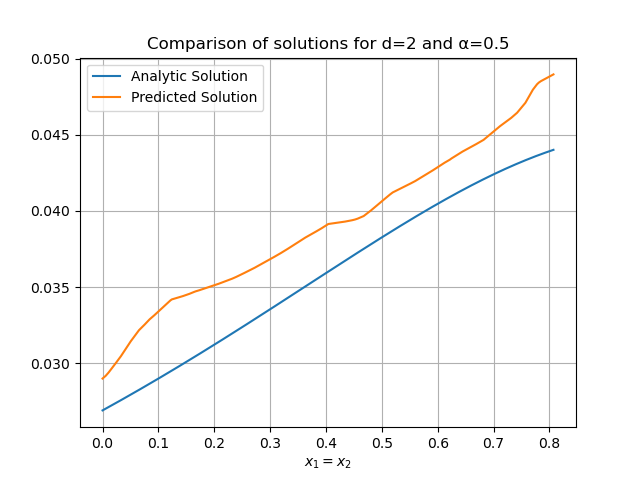}
		\caption{$\alpha = 0.5$}
		\label{fig:ex4d2a05}
	\end{subfigure}
	\begin{subfigure}[b]{0.4\textwidth}
		\centering
		\includegraphics[width=\textwidth]{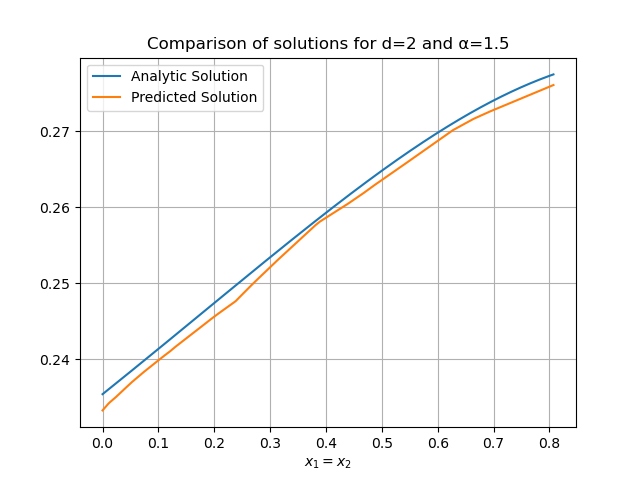}
		\caption{$\alpha = 1.5$}
		\label{fig:ex4d2a15}
	\end{subfigure}
	\begin{subfigure}[b]{0.4\textwidth}
		\centering
		\includegraphics[width=\textwidth]{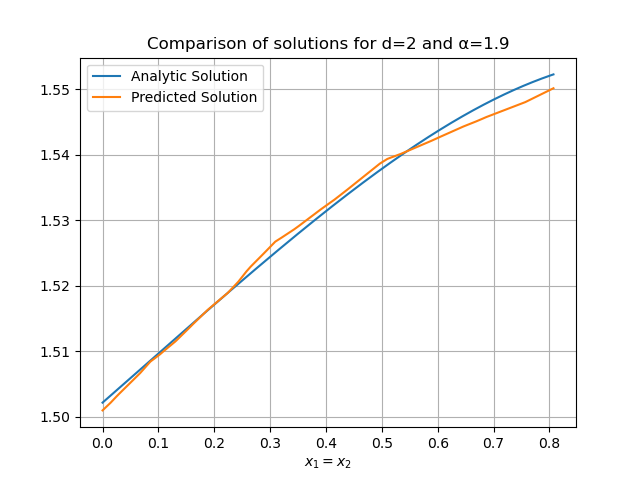}
		\caption{$\alpha = 1.9$}
		\label{fig:ex4d2a19}
	\end{subfigure}
	\caption{Two dimensional comparison of DNN and exact solution when $d=2$.}
\label{fig:ex4-2D-d2}
\end{figure}

\begin{figure}[H]
\centering
\begin{subfigure}[b]{0.4\textwidth}
	\centering
	\includegraphics[width=\textwidth]{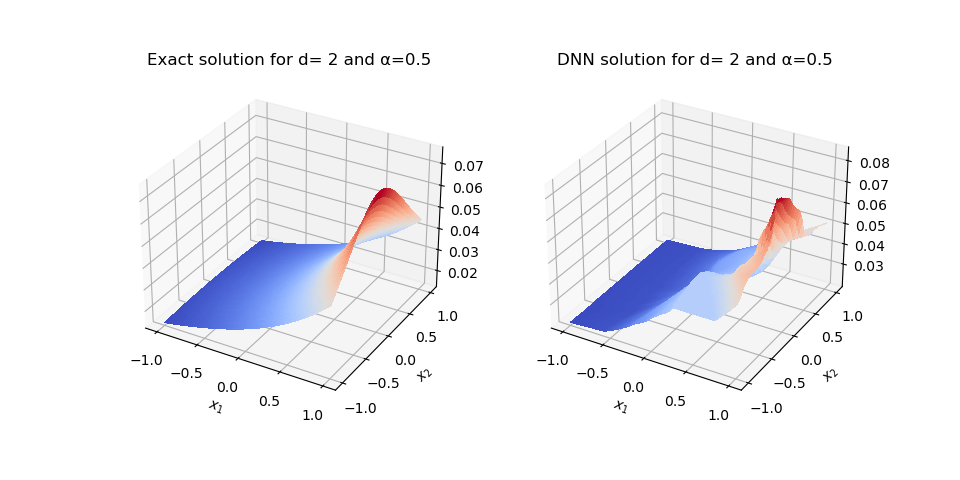}
	\caption{$\alpha = 0.5$}
	\label{fig:ex43Dd2a05}
\end{subfigure}
\begin{subfigure}[b]{0.4\textwidth}
	\centering
	\includegraphics[width=\textwidth]{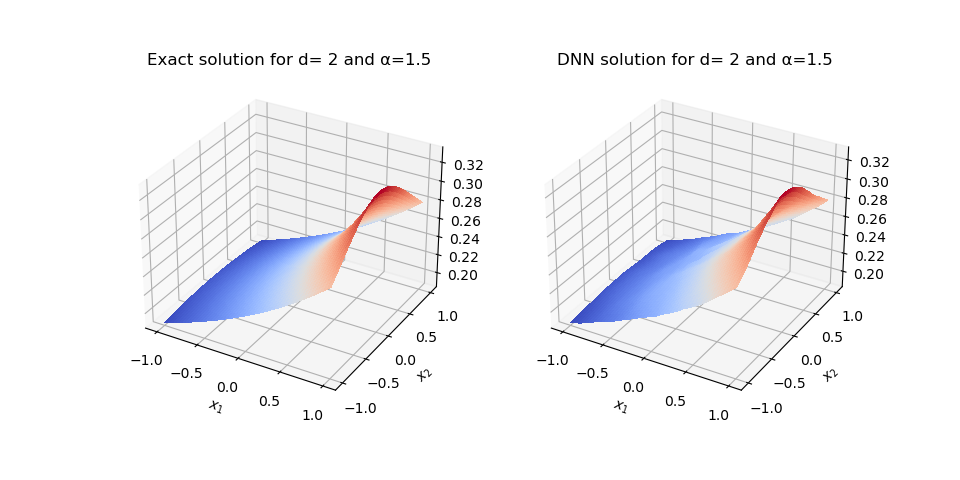}
	\caption{$\alpha = 1.5$}
	\label{fig:ex43Dd2a15}
\end{subfigure}
\begin{subfigure}[b]{0.4\textwidth}
	\centering
	\includegraphics[width=\textwidth]{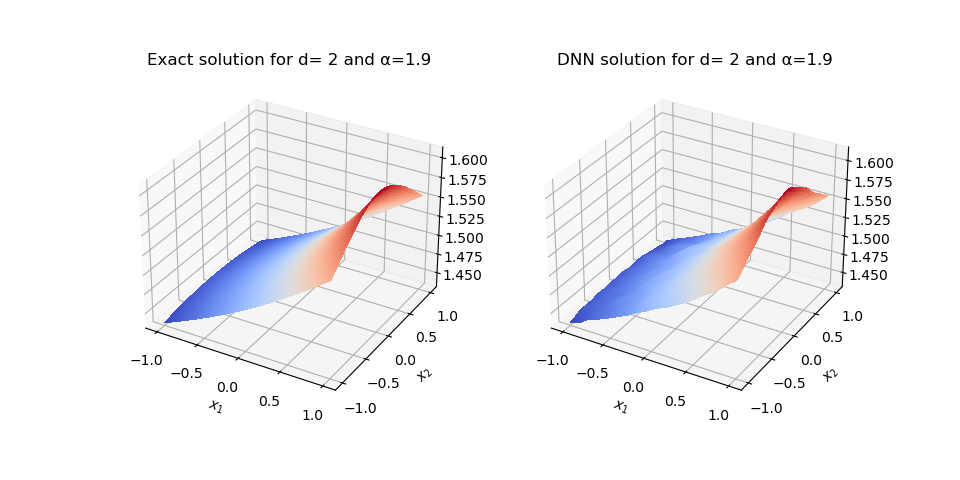}
	\caption{$\alpha = 1.9$}
	\label{fig:ex43Dd2a19}
\end{subfigure}
\caption{Three dimensional comparison of DNN and exact solution for $d=2$. In each Figure, left: exact solution, right: DNN solution.}
\label{fig:ex4-3D-d2}
\end{figure}

Now we study the elapsed time by changing the values of $M$ and $P$ for $\alpha = 0.5,1.5,1.9$. The elapsed times are shown in Tables \ref{tab:ex4a05}, \ref{tab:ex4a15} and \ref{tab:ex4a19}.  

\begin{table}[h]
    \begin{subtable}[h]{0.45\textwidth}
        \centering
	\begin{tabular}{|l|l|l|l|l|}
		\hline
		M\textbackslash{}P & 10     & 100   & 1000    & 2000   \\ \hline
		10                 & 3.64   & 3.78  & 5.36   & 6.54  \\ \hline
		100                & 3.71   & 4.30  & 9.99   & 15.98  \\ \hline
		1000               & 4.31  & 9.21 & 58.56  & 113.46 \\ \hline
		2000               & 4.98  & 15.64 & 111.57  & 219.05 \\ \hline
	\end{tabular}
       \caption{$\alpha = 0.5$.}
       \label{tab:ex4a05}
    \end{subtable}
    \hfill
    \begin{subtable}[h]{0.45\textwidth}
        \centering
	\begin{tabular}{|l|l|l|l|l|}
		\hline
		M\textbackslash{}P & 10     & 100    & 1000    & 2000   \\ \hline
		10                 & 3.71   & 3.83   & 5.84   & 7.62  \\ \hline
		100                & 3.76   & 4.84  & 15.27   & 26.80  \\ \hline
		1000               & 4.61  & 15.62  & 115.40  & 216.40 \\ \hline
		2000               & 5.50  & 26.18  & 216.27  & 423.64 \\ \hline
	\end{tabular}
        \caption{$\alpha = 1.5$.}
        \label{tab:ex4a15}
     \end{subtable}
     
     \medskip
     \begin{subtable}[h]{0.45\textwidth}
        \centering
	\begin{tabular}{|l|l|l|l|l|}
		\hline
		M\textbackslash{}P & 10     & 100     & 1000     & 2000    \\ \hline
		10                 & 3,69   & 4.06    & 8.52    & 12.94   \\ \hline
		100                & 4.02   & 7.49   & 41.39   & 81.52  \\ \hline
		1000               & 6.87  & 41.88  & 387.11  & 768.78 \\ \hline
		2000               & 9.99  & 70.84  & 750.55  & 1486.86 \\ \hline
	\end{tabular}
       \caption{$\alpha = 1.9$.}
       \label{tab:ex4a19}
    \end{subtable}
     \caption{Example 3: Elapsed time in seconds of the SGD algorithm with $d=2$.}
     \label{tab:ex4}
\end{table}

From those values of $M$ and $P$, and for $\alpha = 0.1,0.5,1,1.5,1.9$ we calculate the mean square error and the mean {\color{black}relative} error of the optimal DNN. These errors are in Figure \ref{fig:ex4-MSEP} and \ref{fig:ex4-AEP}.

\medskip
Figure \ref{fig:ex4-MSEP} shows a small mean square error in all settings of $M$, $P$ and $\alpha$, but this small error is a bit tricky, because for small values of $\alpha$, the order of the solution is also small.

\medskip
The errors shown in Figure \ref{fig:ex4-AEP} are very different of errors in Figure \ref{fig:ex4-MSEP}. Figure \ref{fig:ex4-AEP} shows an error near of (or lower than) the order $10^{0}$ for the cases $\alpha=0.1,0.5$. When $\alpha=1$, for larger values of $M$ and $P$ we have an error lower than the order $10^{-2}$, but in the other cases, the error is also near of the order $10^{-1}$. The lowest errors are in the cases of $\alpha = 1.5$ and $\alpha = 1.9$. The error is of the order $10^{-3}$ for large values of $M$ and $P$, and near of the order $10^{-2}$ for small values of $M$ and $P$.

\medskip
\begin{figure}[H]
	\centering
	\begin{subfigure}[b]{0.4\textwidth}
		\centering
		\includegraphics[width=\textwidth]{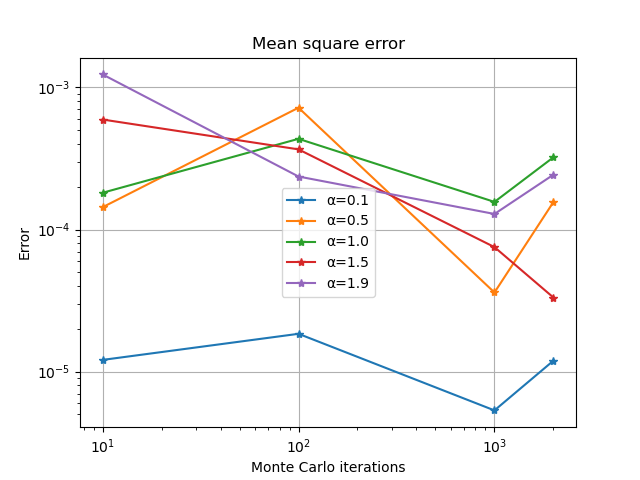}
		\caption{$P = 10$}
		\label{fig:ex4-MSEP10d2}
	\end{subfigure}
	\begin{subfigure}[b]{0.4\textwidth}
		\centering
		\includegraphics[width=\textwidth]{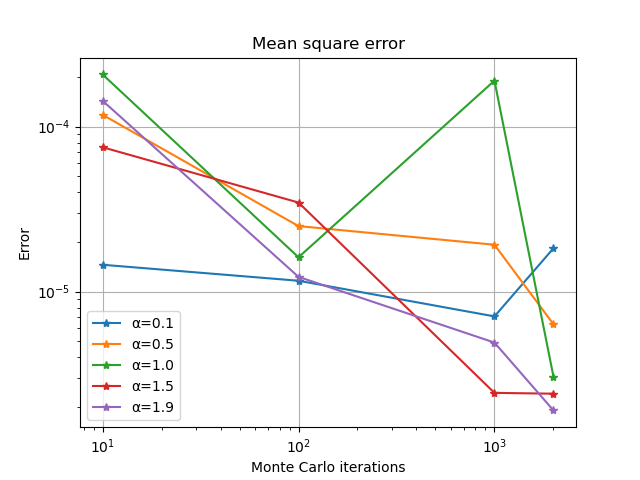}
		\caption{$P = 100$}
		\label{fig:ex4-MSEP100d2}
	\end{subfigure}
	\begin{subfigure}[b]{0.4\textwidth}
		\centering
		\includegraphics[width=\textwidth]{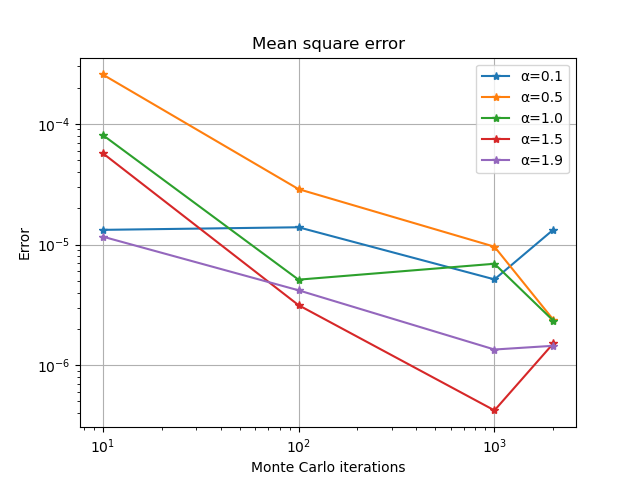}
		\caption{$P = 1000$}
		\label{fig:ex4-MSEP1000d2}
	\end{subfigure}
	\centering
	\begin{subfigure}[b]{0.4\textwidth}
		\centering
		\includegraphics[width=\textwidth]{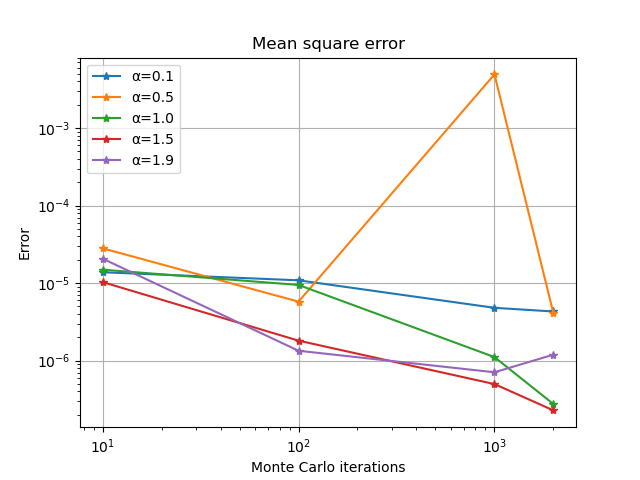}
		\caption{$P = 2000$}
		\label{fig:ex4-MSEP2000d2}
	\end{subfigure}
	\caption{Mean square error against number of Monte Carlo iterations $M$ for 5 values of $\alpha$.}
	\label{fig:ex4-MSEP}
\end{figure}

\medskip
\begin{figure}[H]
	\centering
	\begin{subfigure}[b]{0.4\textwidth}
		\centering
		\includegraphics[width=\textwidth]{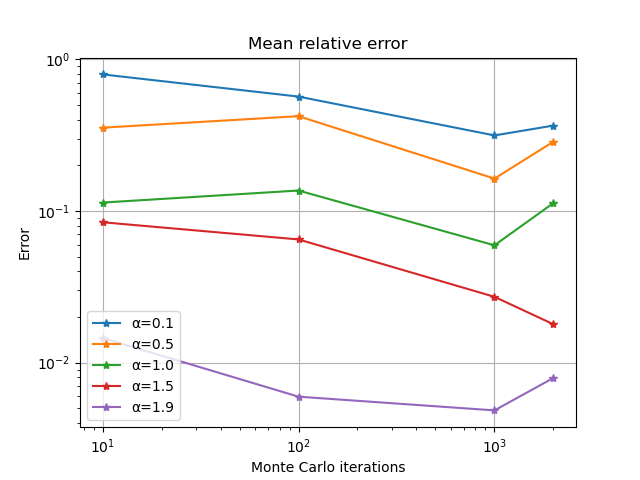}
		\caption{$P = 10$}
		\label{fig:ex4-AEP10d2}
	\end{subfigure}
	\begin{subfigure}[b]{0.4\textwidth}
		\centering
		\includegraphics[width=\textwidth]{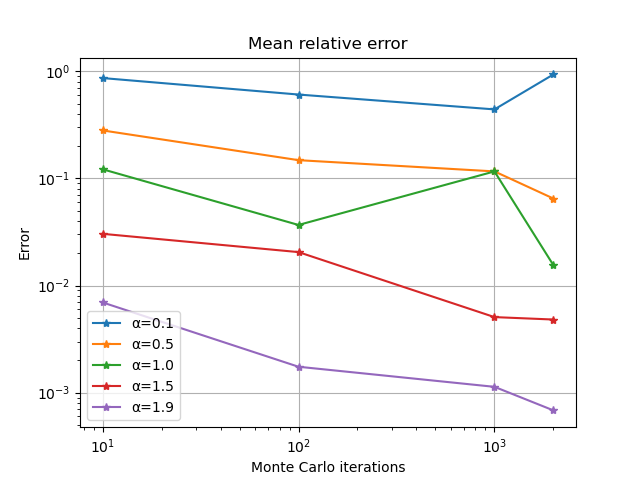}
		\caption{$P = 100$}
		\label{fig:ex4-AEP100d2}
	\end{subfigure}
	\hfill
	\begin{subfigure}[b]{0.4\textwidth}
		\centering
		\includegraphics[width=\textwidth]{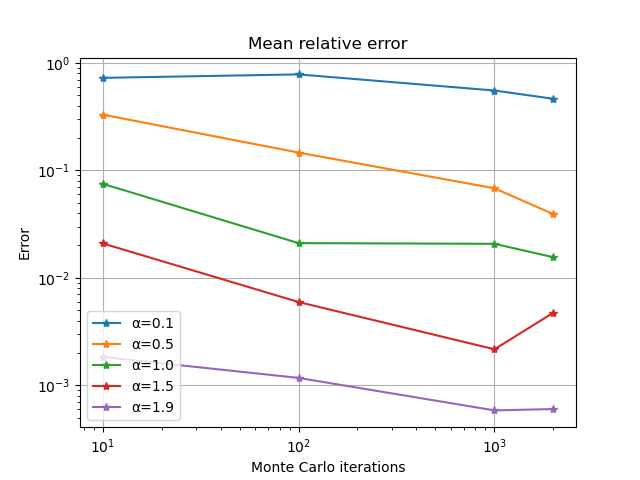}
		\caption{$P = 1000$}
		\label{fig:ex4-AEP1000d2}
	\end{subfigure}
	\centering
	\begin{subfigure}[b]{0.4\textwidth}
		\centering
		\includegraphics[width=\textwidth]{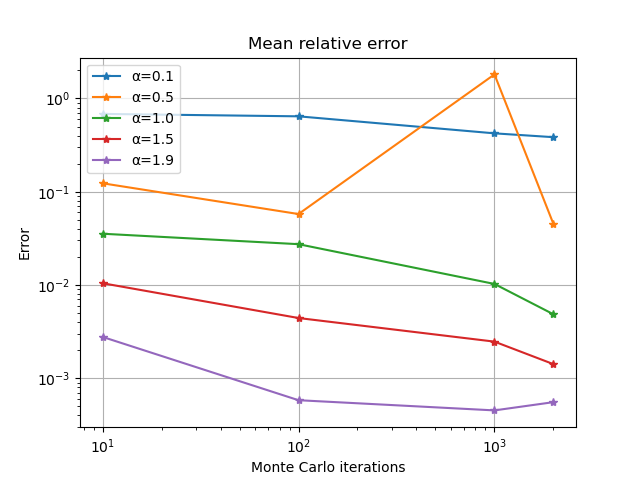}
		\caption{$P = 2000$}
		\label{fig:ex4-AEP2000d2}
	\end{subfigure}
	\caption{Mean {\color{black}relative} error against number of Monte Carlo iterations $M$ for 5 values of $\alpha$.}
	\label{fig:ex4-AEP}
\end{figure}
\subsection{Example 4: A counter example for the algorithm}
Let $d\geq 2$ and $D = B(0,1).$ Consider the following problem
\begin{equation}\label{eq:sol_ex4}
	\centering
	\begin{cases}
		(-\Delta)^{\frac{\alpha}2} u(x) &= 0 \hfill x \in B(0,1),\\
		\hfill u(x) &= \displaystyle \sum_{i=1}^{d} x_i \qquad x \notin B(0,1),
	\end{cases}
\end{equation}
Property 2 in \cite{Ex4} ensures that every harmonic function must be $\alpha$-harmonic for all $\alpha \in (0,2)$. In particular, the function
\[
g(x) = \sum_{i=1}^{d} x_i, \qquad x \notin B(0,1),
\]
is harmonic, and then it is $\alpha$-harmonic. Therefore, for all $\alpha \in (0,2)$, the solution of problem \eqref{eq:sol_ex4} is  
\[
u(x) = \sum_{i=1}^{d} x_i, \qquad x \in \R^d.
\]
For this Example we only consider dimension 2. The performances of the algorithm are similar for different choices of $d$. Figure \ref{fig:ex2-2D-d2} shows a comparison between the optimal DNN and the solution of problem \eqref{eq:sol_ex4} in two dimensions. Figure \ref{fig:ex2-3D-d2} shows the same comparison, but in three dimensions.

\begin{figure}[H]
	\centering
	\begin{subfigure}[b]{0.4\textwidth}
		\centering
		\includegraphics[width=\textwidth]{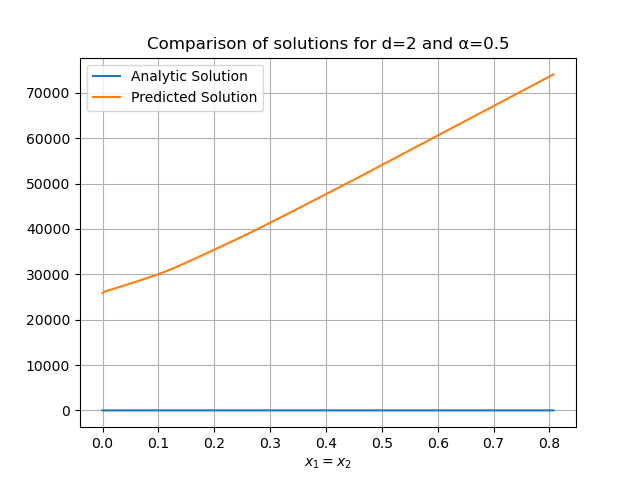}
		\caption{$\alpha = 0.5$}
		\label{fig:ex2d2a05}
	\end{subfigure}
	\begin{subfigure}[b]{0.4\textwidth}
		\centering
		\includegraphics[width=\textwidth]{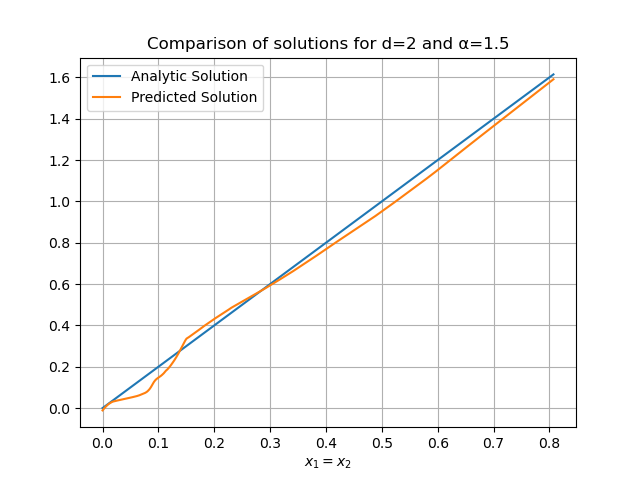}
		\caption{$\alpha = 1.5$}
		\label{fig:ex2d2a15}
	\end{subfigure}
	\begin{subfigure}[b]{0.4\textwidth}
		\centering
		\includegraphics[width=\textwidth]{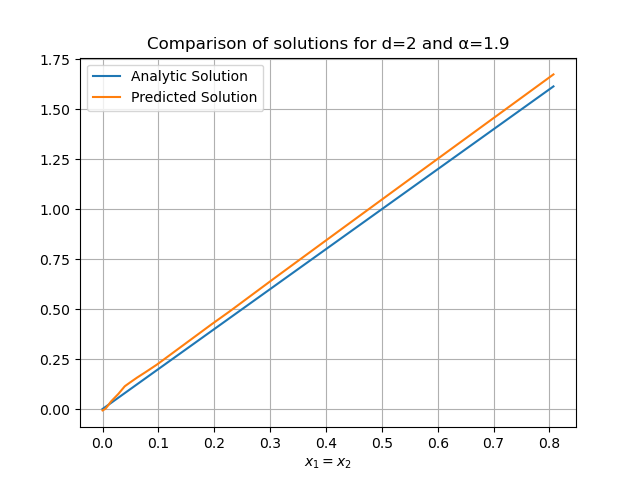}
		\caption{$\alpha = 1.9$}
		\label{fig:ex2d2a19}
	\end{subfigure}
	\caption{Two dimensional comparison of DNN and exact solution when $d=2$.}
\label{fig:ex2-2D-d2}
\end{figure}

\begin{figure}[H]
\centering
\begin{subfigure}[b]{0.4\textwidth}
	\centering
	\includegraphics[width=\textwidth]{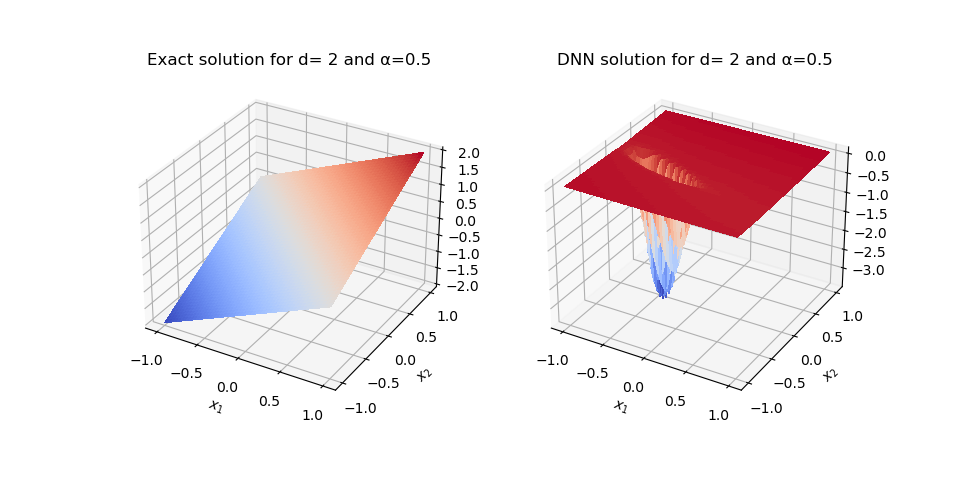}
	\caption{$\alpha = 0.5$}
	\label{fig:ex23Dd2a05}
\end{subfigure}
\begin{subfigure}[b]{0.4\textwidth}
	\centering
	\includegraphics[width=\textwidth]{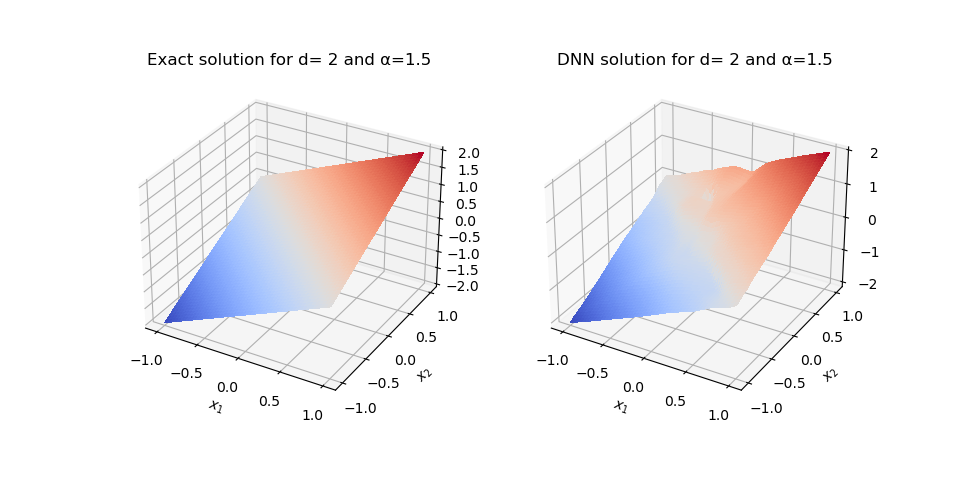}
	\caption{$\alpha = 1.5$}
	\label{fig:ex23Dd2a15}
\end{subfigure}
\begin{subfigure}[b]{0.4\textwidth}
	\centering
	\includegraphics[width=\textwidth]{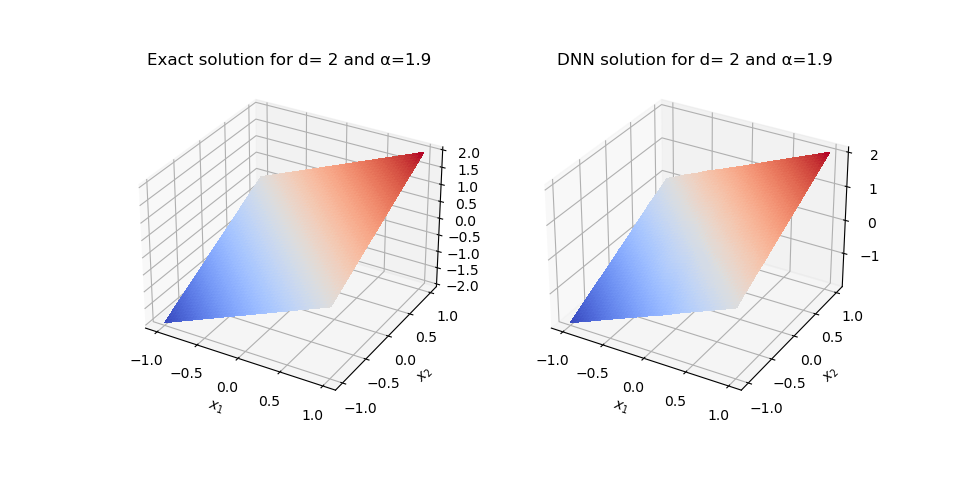}
	\caption{$\alpha = 1.9$}
	\label{fig:ex23Dd2a19}
\end{subfigure}
\caption{Three dimensional comparison of DNN and exact solution for $d=2$. In each Figure, left: exact solution, right: DNN solution.}
\label{fig:ex2-3D-d2}
\end{figure}
\medskip
We can notice from Figures \ref{fig:ex2-2D-d2} and \ref{fig:ex2-3D-d2} that for almost all the settings of $\alpha$, the algorithm does not perform a good approximation of the solution of problem \eqref{eq:sol_ex4}. We only have a well approximation when $\alpha$ is near 2.

\section{Conclusion and discussion}\label{Sec:5}

In this section we talk about conclusions and discussions about {\color{black}our proposed} algorithm and its behavior. 

\medskip
First of all we will highlight the behavior of the WoS process. For small values of $\alpha$, the isotropic $\alpha$-stable process exits far from the boundary $\partial B(0,1)$. This implies that the WoS process exits the set $D$ in a few steps, that means, the value of $N$ is also small. 

\medskip
On the other hand, when $\alpha$ is near 2 (i.e. $\alpha$ is large), the isotropic $\alpha$-stable process exits very close to the boundary $\partial B(0,1)$. This means that the WoS process needs more steps to exit the set $D$. In other words, the value of $N$ is large. This makes sense, because in the case $\alpha = 2$, one can see in \cite{Grohs,AK1} that the WoS process in this case never exits the set $D$, and then the value of $N$ is infinite a.s.

\medskip
In terms of time, we note that the {\color{black}previously} mentioned behavior of the WoS process {\color{black}has} a strong influence in the elapsed time of the SGD algorithm. Indeed, we can see in the different tables on each example that for the same values of $P$ and $M$, the algorithm takes more time for large values of $\alpha$ than for small ones.

\medskip
Notice also that in all the examples, there is a {\color{black}good} approximation for large values of $\alpha$. In the case of small values of $\alpha$, the approximation has some issues, specially in the examples where the boundary condition is non-zero. This problem is caused principally due the same behavior of the WoS pocess: If $\alpha$ is small, the last element in the WoS process, namely $\rho_N$, is a point far away the set $D$, and if the boundary condition is non-zero, the Monte Carlo iterations have the terms involving $g(\rho_N)$. This leads to many errors if a small number of Monte Carlo iterations are considered, that is our case. This {\color{black}has} not happened in \cite{AK1}, because the order of the Monte Carlo iterations in that paper is much higher ($2^{10}$ iterations) compared to the order considered in this paper (maximum 2000 iterations).   

\medskip
We want to emphasize the principal advantage of this numerical method. With this method we can estimate the value of the solution of the fractional PDE in every point of $\R^d$ once the deep neural network is trained, contrary to the case of Monte Carlo, where we need to do a different Monte Carlo realization for each point of $\R^d$, and that is an expensive computation. The same happens in the case of finite differences, where we know the value of the solution only on a grid, and the computational cost is larger the finer grid is. For this method, and in the examples presented, we only needed a small number of Monte Carlo iterations, for a given set of points on $\R^d$, with maximum size of $2000$.

\medskip

\end{document}